\documentclass{article}
\usepackage[utf8]{inputenc}
\usepackage[T1]{fontenc}
\usepackage{amssymb,amsmath,amsthm,amsfonts, epsfig}           %For double arrows
\usepackage{mathtools} % provides enhancements / bug-fixes for amsmath
\usepackage{hyperref}
\usepackage{eucal}
\usepackage[a4paper, bottom=2.5cm]{geometry}
\usepackage[svgnames]{xcolor}
\usepackage{mathptmx} %times font
\usepackage{graphicx}
\usepackage{fancyhdr}
\usepackage{enumerate}

\usepackage[numbers]{natbib} % for more flexible bibtex styles

\usepackage{framed}

\usepackage[normalem]{ulem} % to cross out text in drafts

\usepackage{authblk}

\title{Structure preserving discretization of time-reparametrized Hamiltonian systems with application to nonholonomic mechanics}

\author[1]{Luis C. García-Naranjo}
\affil[1]{Departamento de Matemáticas y Mecánica,
IIMAS-UNAM, Mexico City, MEXICO. \newline \texttt{luis@mym.iimas.unam.mx}}
\author[2]{Mats Vermeeren}
\affil[2]{School of Mathematics, University of Leeds, Leeds, UK. \newline \texttt{m.vermeeren@leeds.ac.uk}}

\date{}

\numberwithin{equation}{section}
\numberwithin{table}{section}
\numberwithin{figure}{section}

\hypersetup{colorlinks=true,linkcolor=violet,citecolor=purple}

\newtheorem{theorem}{Theorem}[section]

\newtheorem{proposition}[theorem]{Proposition}
\newtheorem{corollary}[theorem]{Corollary}

{\theoremstyle{definition}
\newtheorem{definition}[theorem]{Definition}

\newtheorem*{remarks*}{Remarks}

\newtheorem{example}[theorem]{Example}

}

\newtheorem{innercustomgeneric}{\customgenericname}
\providecommand{\customgenericname}{}
\newcommand{\newcustomtheorem}[2]{%
  \newenvironment{#1}[1]
  {%
   \renewcommand\customgenericname{#2}%
   \renewcommand\theinnercustomgeneric{##1}%
   \innercustomgeneric
  }
  {\endinnercustomgeneric}
}

\newcustomtheorem{customhyp}{}

\newcommand{\defn}[1]{{\bfseries\itshape{#1}}}

\def\headcolour{\color{Grey}}
\usepackage{soul}

\def\restr#1{\,\vrule height1.2ex width.4pt
  depth0.8ex\lower0.4ex\hbox{\scriptsize $\,#1$}}

\newcommand{\cN}{\mathcal{N}}

\newcommand{\R}{\mathbb{R}}

\newcommand{\g}{\mathfrak{g}}

\newcommand{\cO}{\mathcal{O}}

\renewcommand{\d}{\mathsf{d}}

\begin{document}

\maketitle

\begin{abstract}
    We propose a discretization of vector fields that are
    Hamiltonian up to multiplication by a positive function
    on the phase space that may be interpreted as a time 
    reparametrization. We prove that our  method is structure preserving in the sense that the discrete flow is interpolated to arbitrary order by the flow of a continuous system possessing 
    the same structure. In particular, our discretization 
    preserves a smooth measure on the phase space to arbitrary order. We present applications to a remarkable class of nonholonomic mechanical systems that allow
    Hamiltonization. 
    To our best knowledge, these results provide the  first
    occurrence in the literature of a measure
    preserving discretization of  measure preserving nonholonomic systems.
    
    
\end{abstract}

\section{Introduction}

The mission of geometric integrators is to develop numerical schemes that preserve the 
geometric structure of differential equations. 
 Geometric integrators often show (near) conservation of the conserved quantities of the continuous system, which is a major benefit for the numerical performance over long timescales. A central role in geometric integration is taken by symplectic or variational integrators. For a broad overview of these concepts and their history we refer to \cite{MarsdenWest, leimkuhler2004simulating, hairer2006geometric}.

This paper deals with the geometric
discretization of differential equations on $\R^{2n}$
of the form
\begin{equation}
\label{eq:itro-conformal}
\frac{ \d q}{\d t}= \mathcal{N}(q,p) \frac{\partial H}{\partial p} (q,p), \qquad \frac{ \d p}{\d t}= -\mathcal{N}(q,p) \frac{\partial H}{\partial q} (q,p),
\end{equation}
where $(q,p)\in \R^n\times \R^n$ and 
$H,\mathcal{N} \in C^\infty(\R^{2n}) $ with $\mathcal{N}>0$.
Namely, we deal with vector fields that are 
Hamiltonian modulo
the multiplication by 
a positive function, or, equivalently, a time reparametrization. In agreement 
with the terminology used in \cite{Marle2012, borisov2008conservation}, we  refer to these 
systems 
as  \defn{conformally Hamiltonian} and to the positive function $\mathcal{N}$ as the
\defn{conformal factor}. It is not hard to see that their flow  preserves the smooth measure $\mu= \mathcal{N}^{-1}\, \d q \wedge \d p$ on 
$\R^{2n}$ and the
Hamilton function $H$ is a first integral (see Section~\ref{S:ConfHam}).

Our main motivation to consider this kind of equations comes from a  remarkable
class of nonholonomic mechanical systems with symmetry, commonly known as 
\defn{Hamiltonizable $G$-Chaplygin systems}, 
whose reduced dynamics 
have this structure (see e.g. \cite{Chaplygin, Stanchenko, FedJov, Ehlers2005,Borisov2007,Fernandez2009,LGN18}, and references therein).

The geometric integration of Equations~\eqref{eq:itro-conformal} has previously been 
considered by Hairer~\cite{hairer1997variable} and 
Reich~\cite{reich1999backward} in connection
with the design of symplectic integrators with variable time step (see also \cite[Section VIII.2]{hairer2006geometric} and \cite[Chapter 9]{leimkuhler2004simulating}),
and also by Fernandez et al~\cite{FernandezBlochOlver2012} with our same motivation of
application to nonholonomic Hamiltonizable $G$-Chaplygin systems. The strategy followed by
these references  to discretize  a trajectory of 
Equations~\eqref{eq:itro-conformal} with 
initial condition $(q_0,p_0)$ is to
replace Equations~\eqref{eq:itro-conformal} by the \defn{altered system}
\begin{equation}
\label{eq:itro-conformal-altered}
\begin{split}
\frac{ \d q}{\d t} &= \mathcal{N}(q,p) \frac{\partial H}{\partial p} (q,p)
+\frac{\partial \mathcal N}{\partial p} (q,p)(H(q,p)-E), \\ \frac{ \d p}{\d t} &= -\mathcal{N}(q,p) \frac{\partial H}{\partial q} (q,p)-\frac{\partial \mathcal N}{\partial q} (q,p)(H(q,p)-E),
\end{split}
\end{equation}
where $E=H(q_0,p_0)$.  The altered Equations~\eqref{eq:itro-conformal-altered} agree
with Equations~\eqref{eq:itro-conformal} on the level set $\{H=E\}$ 
but have the advantage of being Hamiltonian 
with respect to the \defn{altered Hamiltonian}\footnote{The introduction of the altered system and Hamiltonian is sometimes called the Poincaré or (Darboux-)Sundman transformation and has been known for over a century~\cite{levi1906resolution}.}  $K_E(q,p)=\mathcal{N}(q,p)(H(q,p)-E)$, 
and hence may be discretized with a standard  symplectic integrator $\Psi$. 
%{\color{red}
%The details of this construction are discussed in Section \ref{S:ConfHam}.
%}

In this paper we refine the method described above by replacing $E=H(q_0,p_0)$
in Equations~\eqref{eq:itro-conformal-altered} with a more sophisticated choice 
$E=\mathcal{E}(q_0,p_0)$, 
that carries
valuable information about the symplectic integrator $\Psi$ and its backward error analysis. By doing so, we obtain
a ``structure-preserving'' discretization of the conformally Hamiltonian vector field $X$ defined by Equations~\eqref{eq:itro-conformal}. By structure-preserving we mean that the discrete flow is interpolated by the flow of a certain conformally Hamiltonian vector
field $X_{mod}$ on $\R^{2n}$ whose Hamiltonian is the function $\mathcal{E}$ 
(which is close  to $H$) and whose conformal factor is a certain function 
 $\mathcal{N}_{mod}\in C^\infty(P)$ that is close to $\mathcal{N}$. 
 For this reason, we will call $\mathcal{E}$ the \defn{modified conformal Hamiltonian}.
In particular, our discrete scheme preserves the smooth measure $\mu_{mod}=\mathcal{N}_{mod}^{-1}\, \d q\wedge \d p$, which is close to the
the invariant measure $\mu$ of $X$, and has good energy behavior. Our interpolation result
can be understood as  an extension of a well-known result that explains the 
good performance of symplectic integrators in the approximation 
of classical Hamiltonian vector fields \cite{benettin1994hamiltonian}.

\subsection{Application to nonholonomic mechanics}

In recent years there has been a large number of publications (e.g.~\cite{Cortes2001,McPerlm2006,Ferraro_2008,Kobilarov2010,Ferraro_2015} and others)
attempting to develop a geometric discretization of nonholonomic systems that provides an  
extension of variational and symplectic integrators of Hamiltonian systems. 
Such works are usually concerned with general nonholonomic systems whose flow
generally possesses no other invariants besides energy and time-reversibility. The usefulness 
of these integrators is assessed by their energy behavior, their preservation
of the constraints and momentum first integrals (if present), and their performance in simple test problems. There are   many
open questions  in this area; we mention the recent work of       
Modin and Verdier \cite{modin2017} proving that the good performance of  some of these integrators 
is often due to a bias in the choice of the test problems.
 The underlying difficulty is that, in general, the geometry of nonholonomic systems is not as rich and well-understood
as that of Hamiltonian systems.

\paragraph{The discretization of nonholonomic constraints and $G$-Chaplygin systems.}
The constraints of most nonholonomic systems found in applications
are  linear in the velocities and define a sub-bundle $D\subset TM$, where
$M$ is the configuration space. The constraint space $D$ thus has a clear and purely geometric interpretation.
A main difficulty in the geometric  discretization of nonholonomic systems is to 
adequately come up with a ``geometric'' discrete counterpart of $D$. This problem
leads to the notion of ``exact discrete constraint manifold'' which
is a submanifold of $M\times M$ introduced by McLachlan and Perlmutter~\cite[Section 7.1]{McPerlm2006}
(see also the recent 
preprint~\cite{alex2020}). In contrast with the continuous constraint space $D$,
 the exact discrete constraint manifold has more of a  dynamic than
geometric nature, 
 since it is obtained by advancing
points on $M$ (with allowed initial velocities) by the flow of the continuous system.

The difficulty of geometrically discretizing the nonholonomic 
constraints is avoided in our work by restricting our attention to the 
discretization  of the reduced equations of the  so-called nonholonomic {\em $G$-Chaplygin systems}.   For these systems, the dynamics
is equivariant under the action of the Lie group $G$, and projects to the reduced space $D/G$ which is isomorphic to 
$TQ$ where $Q:=M/G$ is the {\em shape space}. This allows us to work with
  the standard discretization $Q\times Q$ of the reduced space $TQ$.  We refer the
  reader to \cite{Stanchenko, Koi,BKMM, CaCoLeMa,FedJov,GNMarr2018} for definition and properties of
  $G$-Chaplygin systems (also known as {\em generalized Chaplygin systems} or
  as the {\em principal or purely kinematic case} in \cite{BKMM}).

\paragraph{Hamiltonizable $G$-Chaplygin systems} The methods developed in this paper provide a discretization of
a remarkable subclass of $G$-Chaplygin systems which possess an extraordinary geometric structure which
allows one to write their reduced equations of motion in the form \eqref{eq:itro-conformal}. As mentioned above,
these are often called Hamiltonizable $G$-Chaplygin systems and their study goes back 
to the classic work of Chaplygin \cite{Chaplygin}.
Examples of nonholonomic systems within this family are the
so-called {\em $\phi$-simple systems}, recently found 
in \cite{LGN18,GNMarr2018} and described in Section~\ref{S:phi-simple}. They include 
the nonholonomic particle, the Veselova problem \cite{veselov1988integrable} and some of its multidimensional generalizations \cite{FedJov},
the  rubber Routh sphere \cite{borisov2008conservation, borisov2013hierarchy} and its multidimensional generalization \cite{LGN18}, the rubber generalization
of a problem of Woronetz \cite{borisov2013hierarchy}, and others.

\paragraph{Measure preservation.} A fundamental property of Hamiltonizable $G$-Chaplygin systems is that  
they possess a smooth invariant measure and the key contribution of our proposed discretization is that it is
measure preserving. To the best of our knowledge,  the problem of finding a measure preserving discretization of 
nonholonomic systems possessing a smooth invariant measure  
had not been considered before. In particular, the work of Fernandez et al~\cite{FernandezBlochOlver2012}, that 
is also concerned with the discretization of Hamiltonizable $G$-Chaplygin systems, does not address this issue.
In fact, our numerical experiments   indicate that the discretization 
proposed in this reference has poor  measure preservation 
properties when compared to our method (see Figures
\ref{fig:measures} and \ref{fig:measure-naive}, and the discussion
in Section~\ref{ss:comparison}).

The reader should be aware that preservation of a smooth measure by a nonholonomic system is an extraordinary property that should not be expected to hold  unless the system is rich in symmetries
(see e.g. the conclusions on the examples treated in~\cite{Fed2015}).
Moreover, 
there are examples of $G$-Chaplygin systems possessing an invariant measure
that do not allow a Hamiltonization (e.g. \cite{jovanovic2019}), and our discretization does not apply to them.

\subsection{Structure of the paper}

We begin by giving a formal definition
and main properties of conformally Hamiltonian systems and the associated
altered systems
in Section~\ref{S:ConfHam}.
We then explain how conformally Hamiltonian systems arise in nonholonomic $\phi$-simple Chaplygin
systems in Section~\ref{S:phi-simple}. In this section
we also introduce the nonholonomic particle as an 
example that will later serve for numerical experiments.
Section~\ref{sec:discretization} is 
the core of the paper. We first present a review of
symplectic integrators and modified equations in 
Section~\ref{sect-symplectic} and then recall the 
discretization of conformally Hamiltonian 
systems used in~~\cite{hairer1997variable,reich1999backward,FernandezBlochOlver2012} in Section~\ref{ss:previous-discretization}. 
The original results of the paper are presented in 
Section~\ref{ss:our-discretization} where
we define the modified conformal Hamiltonian $\mathcal{E}$ in Definition~\ref{def:mod-conf-Ham} and introduce our discretization in Equation \eqref{eqn:Phi_h}. 
Our main result is formulated in Theorem~\ref{Th:Main} that states that our discretization is structure preserving. The measure preservation properties of our
method are presented as Corollary~\ref{cor-discrete-measure}. We
then present some remarks on the
implementation of the method in 
Section~\ref{sec:implementation}. Finally, 
we present numerical experiments in Section~\ref{sec:numerics} and summarize our
conclusions in Section~\ref{s:conclusions}.
The paper also contains Appendix~\ref{appendix} with
expressions of the second order terms
of the power series expansions that are relevant for the
numerical experiments of Section~\ref{sec:numerics}.

\section{Conformally Hamiltonian systems}
\label{S:ConfHam}

In this section we define conformally Hamiltonian systems and state their main properties. As explained in the introduction, in our terminology ``conformally Hamiltonian'' is equivalent to time-reparametrized Hamiltonian systems. The reader is warned  the terminology ``conformal
Hamiltonian'' and  ``conformally symplectic''  appear in the literature with a completely different meaning  (e.g. \cite{mclachlan2001conformal,CCLl2013} and others).

 Throughout this section we work with conformally Hamiltonian 
systems on a general symplectic manifold $(P,\Omega)$. Later, in section
\ref{sec:discretization} we will require $P=\R^{2n}$ and 
$\Omega=\d q \wedge \d p$.

\subsection{Definition and main properties}
\label{SS:Defn&properties}

\begin{definition}
\label{def:confHam}
Let $(P, \Omega)$ be a symplectic manifold. The vector field $X$ on $P$ is called \defn{conformally Hamiltonian} if it satisfies
\begin{equation}
\label{eqn:confHam}
\mathbf{i}_{ X} \Omega = \mathcal{N} \d H
\end{equation}
for functions $H,\mathcal{N}\in C^\infty(P)$, where $\mathcal{N}$ is strictly positive. We say that $H$ is the \defn{Hamiltonian} and $\mathcal{N}$ is the \defn{conformal factor}.
\end{definition}

  In canonical  coordinates $x=(q,p)$ the conformally Hamiltonian vector field $X$ defines
the equations
\begin{equation}
\label{eq:conformal}
\frac{ \d q}{\d t}= \mathcal{N}(q,p) \frac{\partial H}{\partial p} (q,p), \qquad \frac{ \d p}{\d t}= -\mathcal{N}(q,p) \frac{\partial H}{\partial q} (q,p).
\end{equation}

Let $Y$ denote the (canonically) Hamiltonian vector field with the same Hamilton function $H$,\@ i.e. $\mathbf{i}_{Y} \Omega =  \d H$. Then we have $X=\mathcal{N} Y$.
The scaling of $X$ by $\mathcal{N}^{-1}$ that turns it into a Hamiltonian vector
field may be interpreted as a 
time reparametrization. In fact, the vector fields  $X$ and $Y$ have the same trajectories and their flows 
only differ by the speed at which these trajectories  are traversed.  It is common to say that $X$ is Hamiltonian
in the new time $\tau$ that is related to the original time $t$ by $\d t=\mathcal{N}(x)\d \tau$, where
$x\in P$.

We collect the main properties of conformally Hamiltonian systems in the following.
\begin{proposition}
\label{P:properties-confHam}
Let $X$ be a conformally Hamiltonian vector field on the symplectic manifold $(P,\Omega)$, then  
\begin{enumerate}
\item  the Hamiltonian $H$ is constant along the flow of $X$;
\item  We have
\begin{equation*}
\pounds_X \Omega = \d\mathcal{N} \wedge \d H,
\end{equation*}
where $\pounds$ is the Lie derivative operator;
\item The volume form 
\begin{equation*}
\mu = \mathcal{N}^{-1} \Omega^n,
\end{equation*}
is invariant under the flow of $X$, where $2n$ is the dimension of $P$.
\end{enumerate}
\end{proposition}
\begin{proof}  For (i) note that \eqref{eqn:confHam} implies $\d H(X)=\mathcal{N}^{-1}\Omega(X,X)=0$, by skew-symmetry of $\Omega$.

For (ii) we use the standard properties of the Lie derivative to compute
\begin{equation*}
 \pounds_X \Omega =\pounds_{\mathcal{N}Y} \Omega  =\mathcal{N}  \pounds_Y \Omega + \d\mathcal{N}  \wedge \mathbf{i}_{ Y} \Omega
=\d\mathcal{N} \wedge \d H,
\end{equation*}
since 
 $\pounds_Y \Omega =0$ because $Y$ is Hamiltonian.
 
 For (iii) we use Cartan's magic formula and the fact that $\d\mu=0$ to obtain
 \begin{equation*}
\pounds_X \mu =  \d ({\bf i}_X\mu) =\d ({\bf i}_X( \mathcal{N}^{-1} \Omega^n ) )  =\d ({\bf i}_Y( \Omega^n ) ) =\pounds_Y \Omega^n =0,
\end{equation*}
where the last equality follows from Liouville's Theorem.
 \end{proof}

 Items (i) and  (ii) in the above proposition imply that the restriction of $X$ to the level sets of $H$ preserves the pull-back of the symplectic form $\Omega$ to these level sets, a property that also holds for Hamiltonian vector fields.

In order to show that the restriction of $X$ to a level set of $H$ coincides with the restriction of a true Hamiltonian vector field to this level set, we introduce the \defn{altered Hamilton function} $K_E\in C^\infty(P)$ depending parametrically on 
$E\in \R$ by
\begin{equation}\label{eqn:altHam}
K_E(x)= \mathcal{N}(x) (H(x)-E).
\end{equation}
We will sometimes find it convenient to denote $K_E(x)=K(x;E)$.
Note that the zero level set of $K_E$ coincides with the $E$-level set of $H$.

 Fix $E\in \R$ and
denote by $X_{K_E}$ the corresponding Hamiltonian vector field. Namely, $X_{K_E}$ is characterised by
the condition 
\begin{equation}
\label{eqn:def-X_K}
\mathbf{i}_{ X_{K_E}} \Omega = \d{K_E}.
\end{equation}

\begin{proposition}
\label{P:Alt-Ham}
The vector fields $X$ and $X_{K_E}$ agree when restricted to the the level set  $\{H=E\}$ (or, equivalently, to the zero
level set of $K_E$).
\end{proposition}
\begin{proof}
The claim follows from the relation
\begin{equation*}
\mathcal{N}(x) \d H(x)=\d K_E(x),
\end{equation*}
which holds for all $x$ in the level set in question, and the relations \eqref{eqn:confHam} and \eqref{eqn:def-X_K} that 
respectively define $X$
and $X_{K_E}$.
\end{proof}
As mentioned in the introduction, the consideration of the   altered Hamiltonian $K_E$ and its corresponding vector field $X_{K_E}$ is well-known~\cite{levi1906resolution}
and is sometimes referred to as the Poincaré or (Darboux-)\hspace{0pt}Sundman transformation.

\paragraph{Momentum maps and Noether's Theorem}
Suppose that a Lie group $G$ defines a Hamiltonian action on the symplectic manifold $(P,\Omega)$. This means
that $G$ acts by symplectomorphisms and there exists a momentum map
\begin{equation*}
J:M\to \g^*,
\end{equation*}
where $\g$ is the Lie algebra of $G$, satisfying
\begin{equation*}
\d J_\xi = {\bf i}_{\xi_P}\Omega, \quad \mbox{for all $\xi\in \g$},
\end{equation*}
where $J_\xi\in C^\infty(P)$ is defined as $J_\xi(x)=\langle J(x),\xi\rangle$ and $\xi_P$ is the vector field on 
$P$ defined by the infinitesimal action of $\g$ on $P$.

It is well-known that if $H\in C^\infty(P)$ is $G$-invariant then $J$ is constant along the flow of the Hamiltonian
vector field $Y=X_H$. This is commonly referred to as ``Noether's theorem''. The same property holds for the conformally Hamiltonian vector field $X=\mathcal{N}Y$ since, as mentioned before, the trajectories of $X$ and $Y$ coincide.

\section{Application in nonholonomic mechanics}
\label{S:phi-simple}

Conformally Hamiltonian systems as defined by \eqref{def:confHam} arise in the study of certain nonholonomic
systems with symmetry. Concretely, in the so-called $G$-Chaplygin systems, which have received wide attention
in recent years (see
e.g. \cite{Stanchenko,Koi,BKMM,CaCoLeMa,FedJov,Ehlers2005,GNMarr2018} and others).

For these systems the Lie group $G$ acts freely and properly on the configuration manifold $M$ and its lift to $TM$ leaves  the Lagrangian and the constraints invariant. Moreover, the  group orbits are assumed to be transversal and have complementary dimension  to the allowed directions defined by the constraints at each point of the configuration space  (see e.g. \cite{Koi,BKMM} for precise definitions).

The reduced equations of a $G$-Chaplygin system take the form of an unconstrained  forced mechanical system on the
\defn{shape space} $Q=M/G$. In terms of the \emph{reduced Lagrangian} $L:TQ\to \R$ and in local coordinates one has 
\begin{equation}
\label{eq:forced-syst}
\frac{\d}{\d t} \left ( \frac{\partial L}{\partial \dot q} \right ) -  \frac{\partial L}{\partial  q}= F(q,\dot q),
\end{equation}
for a certain  force $F(q,\dot q)$ that is \defn{gyroscopic}: it does not do work along the motion. It in fact satisfies
\begin{equation*}
\langle  F(q,\dot q), \dot q \rangle =0,
\end{equation*}
where $\langle \cdot , \cdot \rangle$ is the pairing between covectors and vectors.
We refer the reader to \cite{Stanchenko,Koi,BKMM} for details.

\subsection[phi-simple Chaplgyin systems]{$\phi$-simple Chaplgyin systems}
\label{SS:phi-simple}

Recently \cite{LGN18,GNMarr2018}, a remarkable class of  $G$-Chaplygin systems was discovered which are 
conformally Hamiltonian. These are the so-called \defn{$\phi$-simple systems}, which, according
to the results of \cite{LGN18,GNMarr2018},  allow the following expression for the gyroscopic 
force $F$:
\begin{equation}
\label{eq:phi-simple}
F(q,\dot q) = \left \langle  \frac{\partial L}{\partial \dot q} , \dot q \right \rangle  \frac{\partial \phi}{\partial  q}
-  \left \langle   \frac{\partial \phi}{\partial  q} , \dot q \right \rangle  \frac{\partial L}{\partial \dot q},
\end{equation}
for a certain function $\phi \in C^\infty (Q)$. This form of the force $F$ is invariant under changes of coordinates.

We now  show that these systems are indeed conformally Hamiltonian.  We begin by defining  the standard
Legendre transformation and Hamiltonian $\mathcal{H}\in C^\infty(T^*Q)$ by
\begin{equation*}
m=  \frac{\partial L}{\partial \dot q}(q,\dot q) , \qquad \mathcal{H}(q,m)=\langle m, \dot q\rangle -L(q,\dot q).
\end{equation*}
As usual, we assume that $L$ is hyper-regular so that the 
 first of these equations may be inverted to express $\dot q$ as a function of 
$q, m$. Performing the usual chain rule calculations, and assuming that \eqref{eq:phi-simple} holds, we rewrite equation \eqref{eq:forced-syst} as the first order system
\begin{equation*}
\dot q=  \frac{\partial \mathcal{H}}{\partial m}(q,m), \qquad \dot m = -  
 \frac{\partial \mathcal{H}}{\partial q}(q,m) + \left \langle m , 
 \frac{\partial \mathcal{H}}{\partial m}(q,m) \right \rangle   \frac{\partial \phi}{\partial  q}  - \left \langle   \frac{\partial \phi}{\partial  q} ,  \frac{\partial \mathcal{H}}{\partial m}(q,m) \right \rangle m,
\end{equation*}
which is equivalent to
\begin{equation}
\label{eqn:phi-simple-aux}
\dot q=  \frac{\partial \mathcal{H}}{\partial m}(q,m), \qquad \exp(-\phi (q)) \frac{\d}{\d t}\left ( \exp(\phi (q)) m  \right )
= -   \frac{\partial \mathcal{H}}{\partial q}(q,m) + \left \langle m , 
 \frac{\partial \mathcal{H}}{\partial m}(q,m) \right \rangle   \frac{\partial \phi}{\partial  q}.
\end{equation}

We now introduce  the rescaled momenta $p$ 
and Hamiltonian  $H(q,p)$   by
\begin{equation*}
p=\exp ( \phi (q)) m, \qquad H(q,p)=\mathcal{H}(q,\exp ( -\phi (q))p).
\end{equation*}
By the chain rule we have
\begin{equation*}
\begin{split}
\frac{\partial H}{\partial q}(q,p)&= \frac{\partial \mathcal{H}}{\partial q}(q,m) - \left \langle  \frac{\partial \mathcal{H}}{\partial m}(q,m),
m\right \rangle     \frac{\partial \phi}{\partial  q}, \\ 
\frac{\partial H}{\partial p}(q,p)&=\exp ( -\phi (q)) \frac{\partial \mathcal{H}}{\partial m}(q,m).
\end{split}
\end{equation*}
Therefore, Equations \eqref{eqn:phi-simple-aux} may be rewritten as the conformally Hamiltonian system
\[
\dot q= \mathcal{N}(q) \frac{\partial H}{\partial p} (q,p), \qquad \dot p= -\mathcal{N}(q) \frac{\partial H}{\partial q} (q,p),
\]
with conformal factor $\mathcal{N}(q) = \exp ( \phi (q))$.

\begin{example}
\label{ex-np}
The nonholonomic particle considered in \cite{BS93} is a test example for nonholonomic mechanics. It concerns
the motion of a particle in $\R^3$ subject to the constraint 
\begin{equation}
\label{eqn:constraint}
\dot z-y \dot x=0.
\end{equation}
We assume that the Lagrangian of the system is of the form
\begin{equation}
\label{eqn:Lag}
\mathcal{L}=\frac{1}{2}(\dot x^2+\dot y^2 +\dot z^2) - U(x,y),
\end{equation}
where $U(x,y)$ is some potential. The Lagrange-d'Alembert principle leads to the equations of motion
\begin{equation*}
\ddot x = - \frac{\partial U}{\partial x} - \lambda y, \quad \ddot y = - \frac{\partial U}{\partial y}, \quad
 \ddot z=\lambda,
\end{equation*}
where $\lambda$ is a Lagrange multiplier. Differentiating the constraint \eqref{eqn:constraint} 
leads to  $\lambda =\dot x\dot y + y \ddot x$, and
so, the equations of motion may be written as
\begin{equation}
\label{eqn:red-eqns}
(1+y^2)\ddot x=  -y\dot y \dot x- \frac{\partial U}{\partial x}, \qquad \ddot y=
  - \frac{\partial U}{\partial y}.
\end{equation}
together with the constraint equation \eqref{eqn:constraint}. Equations \eqref{eqn:red-eqns}
are the reduced equations which, as we will now show, have the form anticipated by \eqref{eq:forced-syst}.  The symmetry 
group is $G=\R$ acting by translations on $z$ and the shape
space $Q=\R^2$ with coordinates $(x,y)$.

The reduced  Lagrangian $L:TQ\to \R$ is obtained by substituting the constraint \eqref{eqn:constraint} into
the Lagrangian $\mathcal{L}$ given by \eqref{eqn:Lag}. One gets
\begin{equation}
\label{eq:Lag-nh-particle}
L(x,y,\dot x, \dot y)=\frac{1}{2} ((1+y^2)\dot x^2+\dot y^2) - U(x,y).
\end{equation}
Its Euler-Lagrange expression is
\[ 
\frac{\d}{\d t} \left ( \frac{\partial L}{\partial \dot q} \right ) -  \frac{\partial L}{\partial  q}
= \begin{pmatrix} (1 + y^2) \ddot{x} + 2 y \dot{x}\dot{y} + \frac{\partial U}{\partial x}\\  \ddot{y} - y \dot{x}^2 + \frac{\partial U}{\partial y} \end{pmatrix} .
\]
Hence the equtions of motion \eqref{eqn:red-eqns} are indeed of the form \eqref{eq:forced-syst} with the gyroscopic force term given by 
\begin{equation*}
F(x,y,\dot x,\dot y)=\begin{pmatrix}  y\dot x \dot y \\  -y \dot x^2 \end{pmatrix}.
\end{equation*}
A direct calculation shows 
that $F$ may be expressed in the form \eqref{eq:phi-simple} with
\begin{equation*}
\phi(x,y)=-\frac{1}{2}\ln (1+y^2).
\end{equation*}

Proceeding as in Section \ref{SS:phi-simple}, we define the momenta and Hamiltonian 
\begin{equation*}
\begin{split}
&m_x= (1+y^2)\dot x, \qquad m_y= \dot y,  \\
& \mathcal{H}(x,y,m_xm_y)= \frac{1}{2} \left ( \frac{m_x^2}{1+y^2}+ m_y^2 \right ) +U(x,y).
\end{split}
\end{equation*}
Next we define the rescaled momenta
\begin{equation*}
p_x=\frac{m_x}{\sqrt{1+y^2}}, \qquad p_y=\frac{m_y}{\sqrt{1+y^2}},
\end{equation*}
and the Hamiltonian in these new variables
\[
H(x,y,p_x,p_y)=  \frac{1}{2} \left (p_x^2+ (1+y^2) p_y^2 \right )+U(x,y).
\]
The analysis in Section  \ref{SS:phi-simple} guarantees that the equations of motion may 
be written in conformally Hamiltonian form
\begin{equation*}
\begin{split}
&\dot x =\frac{1}{\sqrt{1+y^2}} \frac{\partial H}{\partial  p_x}, \qquad 
\dot y =\frac{1}{\sqrt{1+y^2}} \frac{\partial H}{\partial  p_y},  \\ 
 &\dot  p_x =-\frac{1}{\sqrt{1+y^2}} \frac{\partial H}{\partial  x}, 
\qquad\dot  p_y = -\frac{1}{\sqrt{1+y^2}} \frac{\partial H}{\partial y},
\end{split}
\end{equation*}
with conformal factor
\[
\mathcal{N}(y)= \frac{1}{\sqrt{1+y^2}}.
\]
 Explicitly, we have
\begin{equation}
\label{eqn:nh-flow}
\begin{split}
&\dot x =\frac{p_x}{\sqrt{1+y^2}}, \qquad 
\dot y =\sqrt{1+y^2} p_y,  \\ 
 &\dot  p_x =-\frac{1}{\sqrt{1+y^2}} \frac{\partial U}{\partial  x}, 
\qquad\dot  p_y = -\frac{1}{\sqrt{1+y^2}} \left ( yp_y^2 +  \frac{\partial U}{\partial y} \right ).
\end{split}
\end{equation}
Note that the group $\R$ defines a Hamiltonian action on $T^*\R^2$  by translations of $x$. If the potential $U$ is
independent of $x$ then so is the Hamiltonian $H$ and the 
 corresponding momentum $p_x$ is preserved by the flow as predicted by the discussion in Section \ref{SS:Defn&properties}.

The altered Hamilton function $K_E$ is given by
\begin{equation*}
K_E(x,y,p_x,p_y)= \frac{1}{\sqrt{1+y^2}} \left( \frac{1}{2} \left(p_x^2+ (1+y^2) p_y^2 \right)+ U(x,y) - E \right).
\end{equation*}
The corresponding Hamiltonian vector field $X_{K_E}$ is defined by a set of equations which coincide with 
the system \eqref{eqn:nh-flow}, except for the equation for $p_y$ which takes the form
\begin{equation*}
\dot  p_y = -\frac{1}{\sqrt{1+y^2}} \left ( yp_y^2 +  \frac{\partial U}{\partial y} \right )+ (H(x,y,p_x,p_y) - E)\frac{y}{(1+y^2)^{3/2}}.
\end{equation*}
\end{example}
Other examples of $\phi$-simple nonholonomic systems are the multi-dimensional generalizations
of  the 
Veselova problem (with special inertia 
tensor) treated in \cite{FedJov, FedJov2}, the  rubber Routh sphere \cite{GNRubberRouth2019}, and the motion of an axisymmetric rigid body that rolls
without slipping or spinning over a sphere that is fixed in an inertial plane \cite{LGN18}.

\section{Structure preserving discretization}
\label{sec:discretization}

This section contains the main results of the paper.  
We will construct a discretization of \eqref{eq:conformal} based on a symplectic discretization of the Hamiltonian system corresponding to the altered Hamiltonian $ K_E(q,p)= \mathcal{N}(q,p) (H(q,p)-E)$ 
from Proposition \ref{P:Alt-Ham}, where the parameter $E$ will be adjusted according to the initial condition. This approach is reminiscent of \cite{hairer1997variable,reich1999backward} and in particular \cite{FernandezBlochOlver2012}, but we propose a more refined strategy to pick the value of $E$, which will give our integrator a clear geometric structure.

The section is organized as follows. First we will recall some well-known facts on symplectic integrators and modified equations in Section \ref{sect-symplectic}. We then review the construction of  \cite{hairer1997variable,reich1999backward,FernandezBlochOlver2012} in detail
 in Section \ref{ss:previous-discretization}. Our discretization and main results are given in  Section~\ref{ss:our-discretization}.
 We prove that our discrete scheme is very nearly interpolated by a conformally Hamiltonian vector field in Theorem~\ref{Th:Main} and state its
 measure preservation properties in Corollary~\ref{cor-discrete-measure}. Finally, we discuss some
 aspects of the implementation of our method in Section~\ref{sec:implementation}.
 
 For the rest of the paper we will assume that the symplectic manifold $P$ in Definition \ref{def:confHam} equals $P=\R^{2n}=T^*Q$
where $Q=\R^n$, and $\Omega$ is the canonical symplectic form 
$\Omega = \d q \wedge \d p$ where  $x = (q,p)\in \R^n\times \R^n$
are  global linear coordinates.

\subsection{ Symplectic integrators and modified equations}
\label{sect-symplectic}

First we review some well-known concepts in geometric integration. Our presentation is limited to what we need in the present work. For a more exhaustive treatment of this topic we refer to \cite{hairer2006geometric, leimkuhler2004simulating}.

A \defn{consistent numerical integrator}  associates to a vector field $f(x) \partial_x$ on $P$ a map $\Psi_h: P \rightarrow P$ parametrized by a small step size $h > 0$ and satisfying
\[ \forall x \in P: \quad \Psi_h(x) = x + h f(x) + \mathcal{O}(h^2). \]
\emph{Consistent} refers to the fact that the first order term $f(x)$ matches the vector field.
If the map $\Psi_h$ preserves the symplectic form,
\[ \Psi_h^* \Omega = \Omega, \]
when the integrator is applied to a Hamiltonian vector field, then the integrator is called \defn{symplectic}. There are significant benefits to using symplectic integrators for the numerical approximation of Hamiltonian systems, such as the long-time near-conservation of the energy.

\paragraph{Symplectic integrators via variational integrators.}

An effective way to construct symplectic integrators uses the Lagrangian description of mechanics. One of its advantages is that it makes no assumptions on the structure of the Lagrangian, whereas some common symplectic methods require the Hamiltonian to be separable, i.e.\@ of the form $H(q,p) = K(p) + U(q)$. We will use this approach for the examples in Section \ref{sec:numerics}. 

Assume that the Hamiltonian $H: T^*Q \rightarrow \R$ is nodegenerate, i.e.\@ the Hessian matrix $H_{pp}(q,p)$ is everywhere
invertible, then by Legendre transformation we obtain a Lagrange function $L:TQ \rightarrow \R$, such that solutions $(q,p):[0,T] \rightarrow T^*Q$ to the Hamiltonian system project to the stationary curves $q:[0,T] \rightarrow Q$ for the action functional
\[ S[q] = \int_0^T L(q(t),\dot{q}(t)) \,\d t . \]
Consider the \defn{principal action} or \defn{exact discrete Lagrangian}
\[ L_\mathrm{exact}(q_0,q_1,T) = \int_0^T L(q(t),\dot{q}(t)) \,\d t , \]
where $q(t)$ in the right hand side is the unique stationary curve satisfying $q(0) = q_0$ and $q(T) = q_1$. A \defn{variational integrator} is defined by an approximation $L_d: Q \times Q \times \R$ of the exact discrete Lagrangian,
\[ L_d(q_0,q_1,h) = L_\mathrm{exact}(q_0,q_1,h) + \mathcal{O}(h^2) \]
as the step size $h$ tends to zero.
 (Note that $L_\mathrm{exact}(q_0,q_1,h) = \mathcal{O}(h)$, so the $\mathcal{O}(h^2)$ simply indicates that $L_d(q_0,q_1,h)$ and $L_\mathrm{exact}(q_0,q_1,h)$ agree at leading order.)
We then look for discrete curves $(q_0,q_1,\ldots,q_n)$ that are stationary points of the discrete action
\[ S_d(q_0,q_1,\ldots,q_n;h) = \sum_{j=1}^n L_d(q_{j-1},q_j,h) . \]
Such discrete curves are characterized by the equations
\[ \frac{\partial}{\partial q_j} L_d(q_{j-1},q_j,h) + \frac{\partial}{\partial q_j} L_d(q_j,q_{j+1},h) = 0 , \qquad j = 1,\ldots, n-1. \]
On solutions of this second order difference equation we can define the momentum
\[ p_j = \frac{\partial}{\partial q_j} L_d(q_{j-1},q_j,h) = - \frac{\partial}{\partial q_j} L_d(q_j,q_{j+1},h) . \]
The map $\Psi_h: T^*Q \rightarrow T^*Q: (q_j,p_j) \mapsto (q_{j+1},p_{j+1})$ defined by the above equations is well-known to be symplectic. Hence variational integrators are (equivalent to) symplectic integrators.

\paragraph{Modified equations.}

Symplectic integrators nearly conserve energy over long timescales. If the system has symmetries which are respected by the discretization, the same is true for the corresponding Noether integrals. This excellent numerical behavior can be explained using the concept of {\em modified equations}. This is an example of backward error analysis: instead of directly trying to measure the discretization error, we look for a modification of the continuous system that would have been discretized exactly. 

To derive the modified equation, suppose that 
$x(t) = (q(t),p(t))$ is a continuous curve interpolating discrete solutions, 
\[ x(t+h) = \Psi_h(x(t)) . \]
By Taylor expansion we can write this in terms of $x$ and its derivatives at time $t$ only:
\[ x + h \dot{x} + \frac{h^2}{2}\ddot{x} + \ldots = x + h d_1(x) + h^2 d_2(x) \ldots , \]
where the first order term $d_1$ coincides with the right hand side of the original ODE if the integrator is consistent. In the first order we find $\dot{x} = d_1(x) + \mathcal{O}(h)$, which we can use to simplify the series expansion to
\[ h \dot{x} + \frac{h^2}{2} d_1'(x) d_1(x) = h d_1(x) + h^2 d_2(x) + \mathcal{O}(h^2), \]
hence
\[ \dot{x} = d_1(x) + h\left( d_2(x) - \frac{1}{2} d_1'(x) d_1(x) \right) + \mathcal{O}(h^2). \]
Proceeding iteratively we find a differential equation where the right hand side is a power series in $h$:
\[ \dot{x} = f_0(x) + h f_1(x) + h^2 f_2(x) + \ldots \]
with $f_0 = d_1$, $f_1 = d_2(x) - \frac{1}{2} d_1'(x) d_1(x)$, \ldots.
This is the \defn{modified equation} for $\Psi_h$. Formally, solutions to the modified equation interpolate iterations of $\Psi_h$. We say ``formally'' because the power series in the modified equation usually does not converge.
Error bounds can nevertheless be obtained from it by truncating the power series at a suitable point and estimating the truncation error. For instance, solutions to the truncated modified equation $\dot{x} = f_0(x) + h f_1(x) + \ldots + h^\ell f_\ell(x)$ satisfy $\Psi_h(x(t)) = x(t+h) + \cO(h^{\ell+2})$.

A fundamental property of symplectic integrators is that when applied to a Hamiltonian system, the resulting modified equation is again Hamiltonian: there exists a \defn{modified Hamiltonian} $H_{mod}(x;h)$, which is also a power series in $h$, such that
\[ f_0(x) + h f_1(x) + h^2 f_2(x) + \ldots
= \begin{pmatrix} 0 & 1 \\ -1 & 0 \end{pmatrix} \nabla_x H_{mod}(x;h). \]
Just like the modified equation, we cannot expect the power series $H_{mod}(x;h)$ to converge, even for small $h$. However, up to a truncation error of arbitrarily high order in $h$, the modified Hamiltonian is a conserved quantity. This implies that the original Hamiltonian $H(x)$ is very nearly conserved by the symplectic integrator over long timescales. More details can be found in \cite{benettin1994hamiltonian,hairer1997lifespan,reich1999backward} and \cite[Chapter IX]{hairer2006geometric}. A similar theory exists on the level of variational integrators, where a modified Lagrangian can be found \cite{vermeeren2017modified}.

\subsection{Previous discretization of conformally Hamiltonian systems}
\label{ss:previous-discretization}

The discretization of a conformally Hamiltonian system 
\begin{equation}
\label{eq:conformal-sec4.2}
\frac{ \d q}{\d t}= \mathcal{N}(q,p) \frac{\partial H}{\partial p} (q,p), \qquad \frac{ \d p}{\d t}= -\mathcal{N}(q,p) \frac{\partial H}{\partial q} (q,p),
\end{equation}
 that was obtained in
previous works~\cite{hairer1997variable,reich1999backward,FernandezBlochOlver2012}
relies on the  application of a symplectic integrator to the altered system
\begin{equation*}
\frac{ \d q}{\d t} =  \frac{\partial K_E}{\partial p} (q,p), \qquad 
\frac{ \d p}{\d t} = - \frac{\partial K_E}{\partial q} (q,p),
\end{equation*}
which has a Hamiltonian structure and 
where the altered Hamiltonian $K_E$ is defined by~\eqref{eqn:altHam}. 

Let us denote such symplectic integrator by $\Psi_{h,E}$ 
emphasizing the dependence of the altered system
on the parameter $E$ (and where $h$ is the time step as before). For each
pair $(h,E)$ we have a map
\[ \Psi_{h,E}:\R^{2n} \to \R^{2n}: (q_j,p_j) \mapsto (q_{j+1},p_{j+1}). \] 
Given an initial condition $(q_0,p_0)$ for  the system~\eqref{eq:conformal-sec4.2},
the discretization proposed in~\cite{hairer1997variable,reich1999backward,FernandezBlochOlver2012}
is the map $\Psi_{h,E_0}$ where $E_0=H(q_0,p_0)$.

\subsection{Structure preserving discretization of conformally Hamiltonian systems}
\label{ss:our-discretization}

The main contribution of this paper is to improve the discretization of the previous section by refining the choice of $E = E_0$ in $\Psi_{h,E}$ with a more sophisticated one $E =\mathcal{E}(q_0,p_0)$ that carries information about the symplectic integrator $\Psi$ and its backward error analysis. We now proceed to explain how this is done.

First we consider  the modified Hamiltonian corresponding to $\Psi_{h,E}$,  depending parametrically on $E$, 
\begin{equation}\label{eq:DefK}
K_{mod}(q,p;h,E)= K_E(q,p)+h K_1(q,p;E)+h^2 K_2(q,p;E)+\dots .
\end{equation}
We call $K_{mod}$ the \defn{modified altered Hamiltonian}. 
Since $K_{mod}$ is defined as a formal power series, we will often work with its truncation,
\[ 
K_{mod}^{(\ell)} (q,p;h,E)= K_E(q,p)+h K_1(q,p;E)+\dots+h^\ell K_\ell(q,p;E) ,
\]
to make sure we have a well-defined function. The central object in our construction is defined below.

\begin{definition}
\label{def:mod-conf-Ham}
The \defn{modified conformal Hamiltonian} is the formal power series
\begin{equation}\label{eq:Eexp}
\mathcal{E}(q,p;h)=H(q,p)+h\mathcal{E}_1(q,p)+h^2\mathcal{E}_2(q,p)+\dots
\end{equation}
defined implicitly by 
\begin{equation}
\label{eq:DefE}
K_{mod}(q,p;h,\mathcal{E}(q,p;h))=0.
\end{equation}
\end{definition}

The existence of $\mathcal{E}$ is guaranteed by the implicit function theorem, which generalizes to formal power series by proceeding iteratively from leading order to higher orders. Indeed, for $h=0$ the map $\Psi_{0,E}$ is the identity transformation  in $\R^{2n}$
which is  interpolated by the trivial flow of the Hamiltonian $K_{mod}=0$, which is 
obtained by putting $\mathcal{E}(q,p;0)=H(q,p)$. Moreover, we have
\begin{equation*}
\frac{\partial K_{mod}}{\partial E}(q,p;0,H(q,p)) 
= \frac{\partial K_E}{\partial E}(q,p)
= -\mathcal{N}(q,p) \neq 0.
\end{equation*}
Hence, the assumptions of the implicit function theorem are satisfied.

The truncation after order $\ell$ of the modified conformal Hamiltonian will be denoted by
\[
\mathcal{E}^{(\ell)}(q,p;h)=H(q,p)+h\mathcal{E}_1(q,p)+\dots+h^\ell\mathcal{E}_\ell(q,p),
\]
and satisfies 
\[
K_{mod}^{(\ell)}(q,p;h,\mathcal{E}^{(\ell)}(q,p;h)) = \cO(h^{\ell+1}). 
\]

Our \defn{proposed discretization of the conformally Hamiltonian system} \eqref{eq:conformal-sec4.2} is given by the maps $\Phi_h^{(\ell)}:\R^{2n}\to \R^{2n}$, defined by
\begin{equation}
\label{eqn:Phi_h}
\qquad \Phi_h^{(\ell)}(q,p)=\Psi_{h,\mathcal{E}^{(\ell)}(q,p;h)}(q,p),
\end{equation}
for $\ell \in \mathbb{N}$. These maps depend on the modified conformal Hamiltonian and the parameter $\ell$ denotes the order to which it is calculated.  Note that if the symplectic integrator $\Psi_{h,E}$ is a consistent integrator, then so is $\Phi_h^{(\ell)}$.

In view of~\eqref{eq:Eexp} we notice that for $\ell = 0$ our discretization $\Phi_h^{(0)}$ coincides with the one proposed by \cite{hairer1997variable,reich1999backward,FernandezBlochOlver2012} and described in Section~\ref{ss:previous-discretization}.
We will show that by virtue of the higher order terms, the integrator $\Phi_h^{(\ell)}$
is very nearly interpolated by the flow of a conformally Hamiltonian system and, as a consequence, it very nearly preserves a smooth measure. These are our
main results which are rigorously formulated in  Theorem \ref{Th:Main}  and Corollary \ref{cor-discrete-measure} below. Before stating these results precisely, we give a purely formal sketch of the situation.

For the time being we ignore that the power series~\eqref{eq:DefK}, \eqref{eq:Eexp}, defining the modified quantities, usually do not converge. Then we could construct our discretization without any truncations. Let us denote this fictional method by $\Phi_h^{(\infty)}$. Then, the numerical solution defined by $\Phi_h^{(\infty)}$ with initial conditions $(q_0,p_0)$ is exactly interpolated by a solution to the Hamiltonian vector field with Hamiltonian $K_{mod}(q,p;h,E_0)$, where $E_0 = \mathcal{E}(q_0,p_0;h)$.
However, on the level set $M = \{(q,p) \in \R^{2n} \mid \mathcal{E}(q,p;h) = E_0 \}$ this Hamiltonian vector field coincides with the vector field of a conformally Hamiltonian system. Indeed, by differentiating \eqref{eq:DefE}  implicitly with respect to $q$ and $p$ respectively, we find
\begin{equation*}
\begin{split}
 \frac{\partial K_{mod}}{\partial q}(q,p;h,\mathcal{E}(q,p;h)) &= 
 \left ( - \frac{\partial K_{mod}}{\partial E}(q,p;h,\mathcal{E}(q,p;h)) \right ) 
 \frac{\partial \mathcal{E}}{\partial q}(q,p;h), \\
  \frac{\partial K_{mod}}{\partial p}(q,p;h,\mathcal{E}(q,p;h)) &= 
 \left ( - \frac{\partial K_{mod}}{\partial E}(q,p;h,\mathcal{E}(q,p;h)) \right ) 
 \frac{\partial \mathcal{E}}{\partial p}(q,p;h),
\end{split}
\end{equation*}
which upon evaluation on the level set $M$ yields
\begin{equation}
\label{eqn:mod-relation}
\begin{split}
 \frac{\partial K_{mod}}{\partial q}(q,p;h,E_0) &= 
 \left ( - \frac{\partial K_{mod}}{\partial E}(q,p;h,E_0) \right ) 
 \frac{\partial \mathcal{E}}{\partial q}(q,p;h), \\
  \frac{\partial K_{mod}}{\partial p}(q,p;h,E_0) &= 
 \left ( - \frac{\partial K_{mod}}{\partial E}(q,p;h,E_0) \right ) 
 \frac{\partial \mathcal{E}}{\partial p}(q,p;h).
\end{split}
\end{equation}
Hence, still ignoring convergence issues, we conclude
that the numerical solution in question is interpolated by a solution to the conformally Hamiltonian system with (modified) Hamiltonian $\mathcal{E}$ and (modified) conformal factor $\mathcal{N}_{mod}(q,p;h)$ given by
\begin{equation}
\label{eqn:conf-factor}
\mathcal{N}_{mod}(q,p;h) = -\frac{\partial K_{mod}}{\partial E}(q,p;h,E_0). 
\end{equation}
A rigorous statement in terms of truncations of these quantities is made in the theorem below (where we also argue that $\mathcal{N}_{mod}$ is strictly positive). 

\begin{theorem}
\label{Th:Main}
Fix $\ell \in \mathbb{N}$. Our proposed discretization $\Phi_h^{(\ell)}$ of Equation \eqref{eq:conformal-sec4.2} defined by \eqref{eqn:Phi_h}  satisfies
\[ \Phi_h^{(\ell)}(q(t),p(t)) = (q(t+h),p(t+h)) + \cO(h^{\ell+2}), \]
where $(q(t),p(t))$ is any solution of the system of differential equations
\begin{equation}
\label{eq:General2}
\begin{split}
 \dot q &= 
  - \frac{\partial K_{mod}^{(\ell)}}{\partial E}(q,p;h,\mathcal{E}^{(\ell)}(q,p;h))
 \frac{\partial \mathcal{E}^{(\ell)}}{\partial p}(q,p;h), \\
 \dot p &= 
  \frac{\partial K_{mod}^{(\ell)}}{\partial E}(q,p;h,\mathcal{E}^{(\ell)}(q,p;h))
 \frac{\partial \mathcal{E}^{(\ell)}}{\partial q}(q,p;h).
\end{split}
\end{equation}
Moreover, the above system is conformally 
Hamiltonian for sufficiently small $h$. 
\end{theorem}

In other words, the conformally Hamiltonian equation \eqref{eq:General2} can be considered as a truncated modified equation for our discretization \eqref{eqn:Phi_h}.

\begin{proof} 
First note that Equations~\eqref{eq:General2} indeed have the structure of a conformally Hamiltonian system:  the Hamilton function is  $\mathcal{E}^{(\ell)}(q,p;h)$ and the conformal factor is
\begin{equation}
\label{Nmodl}
\mathcal{N}_{mod}^{(\ell)}(q,p;h) = -\frac{\partial K_{mod}^{(\ell)}}{\partial E} \big(q,p;h,\mathcal{E}^{(\ell)}(q,p;h) \big).
\end{equation}
We only need to show that this conformal factor is positive. For this we note that $K_{mod}^{(\ell)}(q,p;0,E)=K_E(q,p)=\mathcal{N}(q,p)(H(q,p)-E)$ and therefore $\mathcal{N}_{mod}^{(\ell)}(q,p;0) = \mathcal{N}(q,p) > 0$.
Thus, by continuity, and possibly restricting  $(q,p)$ to a compact subset of $\R^{2n}$, we conclude that $\mathcal{N}_{mod}^{(\ell)}(q,p;h)$ is 
indeed positive for small $h$.

Consider now a solution $(q(t),p(t))$ to Equation \eqref{eq:General2}. Then by item $(i)$ of Proposition \ref{P:properties-confHam}, the modified conformal Hamiltonian $\mathcal{E}^{(\ell)}(q(t),p(t);h)$ is a constant, which we denote by $E_0$. Hence this solution $(q(t),p(t))$ satisfies
\[ \dot q =  - \frac{\partial K_{mod}^{(\ell)}}{\partial E}(q,p;h,E_0)
 \frac{\partial \mathcal{E}^{(\ell)}}{\partial p}(q,p;h), \qquad 
\dot p =\frac{\partial K_{mod}^{(\ell)}}{\partial E}(q,p;h,E_0)
 \frac{\partial \mathcal{E}^{(\ell)}}{\partial q}(q,p;h). \]
Comparing this to the formal power series equation \eqref{eqn:mod-relation}, and truncating after the $h^\ell$-term, we see that $(q(t),p(t))$ satisfies
\[ \dot q = \frac{\partial K_{mod}^{(\ell)}}{\partial p}(q,p;h,E_0) + \cO(h^{\ell+1}), \qquad 
\dot p =- \frac{\partial K_{mod}^{(\ell)}}{\partial q}(q,p;h,E_0) + \cO(h^{\ell+1}). \]
In other words, $(q(t),p(t))$ satisfies the modified equation for $\Phi_h^{(\ell)}$ with a defect of order $\cO(h^{\ell+1})$, which implies that the local error between it and the numerical solution is $\cO(h^{\ell+2})$. (Note that we can pass from $\cO(h^{\ell+1})$ in the differential equation to $\cO(h^{\ell+2})$ in the local error, because we consider solutions over a time interval of length $h$.)
\end{proof}

Associated to the modified conformal factor \eqref{eqn:conf-factor} there is a modified measure $\mu_{mod} = \mathcal{N}_{mod}^{-1} \Omega^n$. Once again this is a formal power series. We denote 
\[ \mu_{mod}^{(\ell)} = \left( \mathcal{N}_{mod}^{(\ell)}\right)^{-1} \Omega^n , \]
where $\mathcal{N}_{mod}^{(\ell)}$ is defined in Equation \eqref{Nmodl}.
The modified measure is very nearly conserved by our discretization.

\begin{corollary}
\label{cor-discrete-measure}
Over any fixed time interval $[0,T]$, our proposed discretization $\Phi_h^{(\ell)}$, $\ell \in \mathbb{N}$ preserves the modified measure $\mu_{mod}^{(\ell)}$ up to an error of order $\cO(h^{\ell+1})$ 
in the 
following sense: if $K\subset \R^{2n}$ is a compact set and 
 $N \in \mathbb{N}$ is such that $Nh \in [0,T]$, then
 \begin{align*}
\int_{\big(\Phi_h^{(\ell)}\big)^N(K) } \mu_{mod}^{(\ell)} = \int_K  \mu_{mod}^{(\ell)} + \cO(h^{\ell+1}).
\end{align*}
\end{corollary}

\begin{proof}
Let $\varphi_t^{(\ell)}$ be the flow of conformally Hamiltonian equation \eqref{eq:General2} over a time interval of length $t$. By Theorem \ref{Th:Main}, the discrete map satisfies $\Phi_h^{(\ell)}(q,p) = \varphi_h^{(\ell)}(q,p) + \mathcal{O}(h^{\ell+2})$. As the flow of a conformally Hamiltonian system, $\varphi_t^{(\ell)}$ preserves the corresponding measure: $\left(\varphi_t^{(\ell)} \right)^* \mu_{mod}^{(\ell)} = \mu_{mod}^{(\ell)}$. It follows that
\begin{align*}
\int_K \left(\Phi_h^{(\ell)} \right)^* \mu_{mod}^{(\ell)} - \int_K  \mu_{mod}^{(\ell)} =
\int_K \left(\varphi_h^{(\ell)} + \cO(h^{\ell+2}) \right)^* \mu_{mod}^{(\ell)} - \int_K  \mu_{mod}^{(\ell)}
&= \cO(h^{\ell+2}),
\end{align*}
for any compact domain $K$. (Compactness ensures that the $\cO$-term can be pulled out of the integral.)
Hence for all $N \in \mathbb{N}$ such that $Nh \in [0,T]$ there holds
\begin{align*}
\int_K \left( \left(\Phi_h^{(\ell)} \right)^N \right)^* \mu_{mod}^{(\ell)} - \int_K  \mu_{mod}^{(\ell)} 
&=
\int_K \left( \left(\Phi_h^{(\ell)} \right)^* \right)^N \mu_{mod}^{(\ell)} - \int_K  \mu_{mod}^{(\ell)} \\
&= \sum_{j = 1}^N \left[ \int_K  \left( \left(\Phi_h^{(\ell)} \right)^* \right)^j \mu_{mod}^{(\ell)} - \int_K  \left( \left(\Phi_h^{(\ell)} \right)^* \right)^{j-1} \mu_{mod}^{(\ell)} \right] \\
&= \sum_{j = 1}^N \left[ \int_{\Phi_{(j-1)h}^{(\ell)}(K)}  \left(\Phi_h^{(\ell)} \right)^* \mu_{mod}^{(\ell)} - \int_{\Phi_{(j-1)h}^{(\ell)}(K)} \mu_{mod}^{(\ell)} \right] \\
&= \sum_{j = 1}^N \cO(h^{\ell+2}) 
= \cO(h^{\ell+1}). \qedhere
\end{align*}
\end{proof}

\subsection{ Implementation of the integrator}
\label{sec:implementation}

The definition of the map $\Phi_h^{(\ell)}$ requires an evaluation of the modified conformal Hamiltonian $\mathcal{E}$. In practice, it may be convenient to calculate it only once on the initial values $(q_0,p_0)$ and use $\Psi_{h,E}$ with $E = \mathcal{E}(q_0,p_0)$ for all steps. This avoids the potentially expensive calculation of $\mathcal{E}$ at each step in time.
Since $\mathcal{E}$ is the conformal Hamiltonian from Theorem \ref{Th:Main}, it is conserved up to a local error of order $\cO(h^{\ell+2})$. Therefore the results of Theorem \ref{Th:Main} and Corollary \ref{cor-discrete-measure} also apply to $\Psi_{h,\mathcal{E}^{(\ell)}(q_0,p_0)}$. In particular, $\Psi_{h,\mathcal{E}^{(\ell)}(q_0,p_0)}$ is measure preserving up to an error of order $\cO(h^{\ell+1})$ over a fixed time interval. 
In the numerical experiments presented below, we use $\Psi_{h,\mathcal{E}^{(\ell)}(q_0,p_0)}$. In a slight abuse of notation we will keep writing $\Phi_h^{(\ell)}$ to reference the nearly identical map $\Psi_{h,\mathcal{E}^{(\ell)}(q_0,p_0)}$.

To close this section, we present Table \ref{tab:Hamiltonians} with a  summary of the different Hamilton functions that occur in our construction.

\begin{table}[h]
    \centering
    \vspace{5mm}
    \renewcommand*{\arraystretch}{1.5}
    \begin{tabular}{|r c c l|}
        \hline
        Eqn. & Notation & Name & Equations of motion \\\hline
        \multicolumn{4}{|c|}{\textbf{Original system}} \\
        \eqref{eqn:confHam} & $H$ & Conformal Hamiltonian &  $\mathbf{i}_{ X} \Omega = \mathcal{N} \d H$ \\
        \eqref{eqn:altHam} & $K_E = \mathcal{N}(H-E) $ & Altered Hamiltonian & $\mathbf{i}_{ X} \Omega = \d K_E$ on $\{ K_E = 0 \}$ \\\hline
        \multicolumn{4}{|c|}{\textbf{Modified system, interpolating numerical solutions}} \\
        \eqref{eq:DefK} & $K_{mod}$ & Modified altered Hamiltonian &  $\mathbf{i}_{X_{mod}} \Omega = \d K$ on $\{ K_{mod} = 0 \}$ \\
        & $K_{mod}^{(\ell)}$ & Truncated modified altered Hamiltonian & \\
        \eqref{eq:DefE} & $\mathcal{E}$ & Modified conformal Hamiltonian & $\mathbf{i}_{X_{mod}} \Omega = \mathcal{N}_{mod} \d \mathcal{E}$ \\
        & $\mathcal{E}^{(\ell)}$ & Truncated modified conformal Hamiltonian &  \\\hline
    \end{tabular}
    \caption{Overview of all relevant Hamiltonians}
    \label{tab:Hamiltonians}
\end{table}

\section{Numerical experiments}
\label{sec:numerics}

We now apply our discretization to 
 Example \ref{ex-np} of the nonholonomic particle. We first give an outline
 of its implementation 
 using variational integrators  in the presence of an arbitrary
 potential $U(x,y)$  in Section~\ref{ss:disc-nh-particle}. 
 We then present numerical results for the  harmonic potential
 $U(x,y)=\frac{1}{2}(x^2+y^2)$ 
 in Section~\ref{ss:numerics-harmonic-potential} and the free nonholonomic particle ($U(x,y)=0$) in Section~\ref{ss:numerics-free}.

Whenever we write $\Phi_h^{(\ell)}$ in this section, the numerical implementation uses
$\Psi_{h,E}$ with $E=\mathcal{E}^{(\ell)}(q_0,p_0)$ as explained in Section \ref{sec:implementation}. The terms in the power series expansion of  $K_{mod}$, $\mathcal{E}$ and $\mathcal{N}_{mod}$ that are necessary 
for our analysis were  obtained using computer algebra in SageMath and the numerical experiments were implemented in python. All code is available at \cite{vermeeren2020code}.

\subsection{Geometric discretization of the nonholonomic particle}
\label{ss:disc-nh-particle}

Consider the nonholonomic particle in a potential introduced in Example \ref{ex-np}, with equations of motion~\eqref{eqn:nh-flow}. As explained in Section~\ref{SS:phi-simple}, the system is conformally Hamiltonian with Hamiltonian $H$ and conformal
factor $\mathcal{N}$ given by
\begin{equation*}
H(x,y,p_x,p_y)=\frac{1}{2}\left ( p_x^2 + (1+y^2)  p_y^2 \right) + U(x,y), \qquad \mathcal{N}(y)=\frac{1}{\sqrt{1+y^2}}.
\end{equation*}
The corresponding altered Hamiltonian is
\begin{equation*}
 K_E (x,y, p_x,  p_y)= \frac{1}{\sqrt{1+y^2}}(H(x,y,p_x,p_y)-E).
\end{equation*}

Our method involves the construction of a symplectic integrator $\Psi_{h,E}$ for the altered system
\begin{equation*}
 \dot x = \frac{\partial K_E}{\partial p_x}, \quad  \dot y = \frac{\partial K_E}{\partial p_y}, \quad
 \dot p_x =- \frac{\partial K_E}{\partial x}, \quad  \dot p_y = -\frac{\partial K_E}{\partial y}.
\end{equation*}
We find it convenient to do this using variational integrators, as in Section~\ref{sect-symplectic}. With this in mind we consider the 
Legendre transformation of $K_E$ above and introduce the 
 {\em altered Lagrangian} 
\[
\Lambda_E(x,y, \dot{x},\dot{y})=  \frac{1}{\sqrt{1+y^2}} \left( \frac{1}{2} (1+y^2) \dot{x}^2 + \frac{1}{2} \dot{y}^2 - U(x,y)+ E \right) =\mathcal{ N}(y)(L(x,y, \dot{x},\dot{y})+E),
\]
where the Lagrangian $L$ is given by~\eqref{eq:Lag-nh-particle}.
For the discrete Lagrangian we choose one of the following five combinations of midpoint (M) and trapezoidal (T) quadrature of the principal action $\int_0^h \Lambda_E \,\d t =\int_0^h \mathcal{N}(L+E) \,\d t$:
\begin{equation}\label{Ldisc}
    \begin{split}
    & \Lambda_{M}(x_0,y_0,x_1,y_1; E) = \cN \left( \frac{y_0 + y_1}{2} \right)  \left( L \left( \frac{x_0+x_1}{2}, \frac{y_0+y_1}{2}, \frac{x_1-x_0}{h}, \frac{y_1-y_0}{h} \right) + E \right),
    \\
    & \Lambda_{MT}(x_0,y_0,x_1,y_1; E) \\
    &\qquad = \cN \left( \frac{y_0 + y_1}{2} \right)  \left( \frac{1}{2} L \left( x_0, y_0, \frac{x_1-x_0}{h}, \frac{y_1-y_0}{h} \right) + \frac{1}{2} L \left( x_1, y_1, \frac{x_1-x_0}{h}, \frac{y_1-y_0}{h} \right) + E \right),
    \\
    & \Lambda_{TM}(x_0,y_0,x_1,y_1; E) = \frac{ \cN(y_0) + \cN(y_1)}{2}  \left( L \left( \frac{x_0+x_1}{2}, \frac{y_0+y_1}{2}, \frac{x_1-x_0}{h}, \frac{y_1-y_0}{h} \right) + E \right),
    \\
    & \Lambda_{TT}(x_0,y_0,x_1,y_1; E)  \\
    &\qquad = \frac{ \cN(y_0) + \cN(y_1)}{2}
    \left( \frac{1}{2} L \left( x_0, y_0, \frac{x_1-x_0}{h}, \frac{y_1-y_0}{h} \right) + \frac{1}{2} L \left( x_1, y_1, \frac{x_1-x_0}{h}, \frac{y_1-y_0}{h} \right) + E \right),
    \\
    & \Lambda_{T}(x_0,y_0,x_1,y_1; E)  \\
    &\qquad = 
    \frac{1}{2} \cN(y_0) \left( L \left( x_0, y_0, \frac{x_1-x_0}{h}, \frac{y_1-y_0}{h} \right) + E \right) + \frac{1}{2} \cN(y_1) \left( L \left( x_1, y_1, \frac{x_1-x_0}{h}, \frac{y_1-y_0}{h} \right) + 
    E \right).
    \end{split}
\end{equation}

It turns out that, for any potential $U$, the modified altered Hamiltonians for these discretizations only contain terms of even order in $h$:
\[ K_{mod}(q,p;h,E)= K_E(q,p)+h^2 K_2(q,p;E)+h^4 K_4(q,p;E)+\dots . \]
The reason for this is that the discretizations \eqref{Ldisc} are invariant under the transformaton $(x_0,y_0) \leftrightarrow (x_1,y_1)$. As a consequence, the truncated modified altered Hamiltonians satsify $K_{mod}^{(2\ell+1)} = K_{mod}^{(2\ell)}$. The same holds for $\mathcal{E}$ and $\mathcal{N}_{mod}$, hence we have that 
\[ \Phi_h^{(2\ell+1)} = \Phi_h^{(2\ell)}.\]
Recall that throughout this section we use $\Phi_h^{(\ell)}$ to denote the method $\Psi_{h,E}$ with $E=\mathcal{E}^{(\ell)}(q_0,p_0)$, as explained in Section \ref{sec:implementation}.

\subsection{Nonholonomic particle in a harmonic potential}
\label{ss:numerics-harmonic-potential}

We present the results of the numerical experiments for the potential $U(x,y) = \frac{1}{2}(x^2 +  y^2)$. We first investigate measure preservation and energy behavior of the method $\Phi_h^{(4)} = \Phi_h^{(5)} $, constructed for each of the five  discretizations in~\eqref{Ldisc}.
The results are presented in Sections~\ref{sss:measure-harmonic} and \ref{sss:energy}. Our analysis requires the calculation of the fourth order expansion of the corresponding power series for $K_{mod}$, $\mathcal{E}$ and $\mathcal{N}_{mod}$. The second order terms are listed in Appendix~\ref{sec:harmonic-formulas} while the third order terms vanish. We do not give explicit expressions of the fourth order terms because of their length, but these  may be found in our code \cite{vermeeren2020code} (along with the computer algebra tools to derive them). Finally, in Section \ref{ss:comparison} we illustrate
how the method $\Phi_h^{(0)}$ used  in~\cite{hairer1997variable,reich1999backward,FernandezBlochOlver2012} does not enjoy the nice
measure preservation properties of $\Phi_h^{(4)}$.

\begin{figure}[t]
\includegraphics[width=\linewidth]{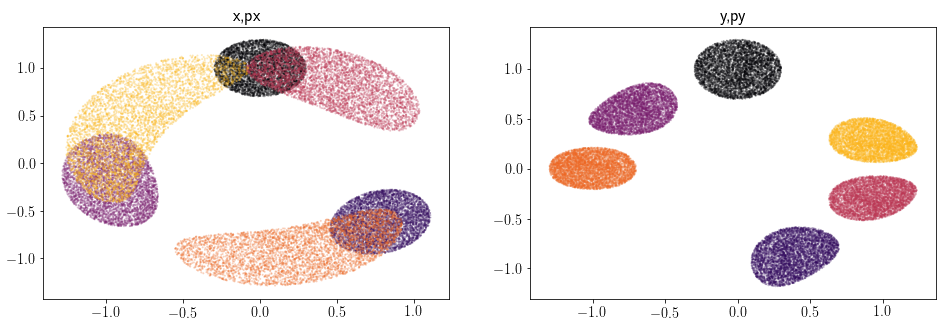}
\caption{Evolution of a spherical cloud (black) of $5000$ points with radius $0.3$ at times $0$, $2.75$, $\ldots$, $13.75$, projected to the $(x,p_x)$ and $(y,p_y)$ planes, using a high-accuracy method. The longer the time, the lighter the color.}
\label{fig:clouds}
\end{figure}

\subsubsection{Measure preservation}
\label{sss:measure-harmonic}

To investigate measure preservation numerically, we compute the trajectories of a point cloud using the map $\Phi_h^{(4)}$, with step size $h=0.25$. The initial points lie on a $3$-sphere centered around $(x,y,p_x,p_y) = (0,0,1,1)$. An illustration of the evolution of such  point cloud is given in Figure \ref{fig:clouds}. From this figure we can clearly see that the flow is not symplectic with respect to $\Omega=\d x \wedge \d p_x + \d y \wedge \d p_y$. If that were the case, the sum of the areas of the projections on the $(x,p_x)$ and $(y,p_y)$ planes would be constant. Note that there was no reason to expect a symplectic flow and that the lack of symplecticity does not imply anything about measure preservation. 

We want to study the volume, with respect to the relevant measures, of the image of the ball bounded by this 3-sphere under the flow over time. In the top-left panel of Figure \ref{fig:measures} we show the evolution of the volume of the convex hull enclosing the point cloud with respect to the measure $\mu_0 = \cN^{-1} \Omega^2$. This is the measure preserved by the continuous system. We see that the reference (high-accuracy) solution nicely preserves this volume until about $t=20$. This is when the region of phase space that we are tracking starts to be non-convex, so by plotting the volume of the convex hull we overestimate the actual volume of the region. The numerical solutions, however, show an oscillating  behaviour. This is because the discretizations do not preserve $\mu_0$ but rather the modified measures $\cN_{mod}^{-1} \Omega^2$, with densities given in terms of  the modified conformal factor $\cN_{mod}$. 

In the top-right panel of Figure \ref{fig:measures} we show the volume evolution, with respect to the first approximation of the modified measure,  $\mu_{mod}^{(2)} = \left(\cN_{mod}^{(2)} \right)^{-1} \Omega^2$, obtained using the conformal factor  $\cN_{mod}^{(2)} = \cN + h^2 \cN_2$. 
The volume with respect to the next approximation, $\mu_{mod}^{(4)} = \left(\cN_{mod}^{(4)} \right)^{-1} \Omega^2$, is shown in the bottom panel of Figure~\ref{fig:measures}. 
As expected, we see a decrease in the magnitude of the oscillations as the order $\ell$ of $\mu_{mod}^{(\ell)}$ increases. Note that, at least before $t=20$,  the amplitude of the fluctuations on 
the last graph is about a factor
of $2\cdot10^{-4}$ of the initial volume. This is in agreement
with Corollary~\ref{cor-discrete-measure} that predicts
preservation of $\mu_{mod}^{(4)} = \mu_{mod}^{(5)}$ by the integrator
$\Phi_h^{(4)}=\Phi_h^{(5)}$ up to order 
 $h^6\approx 2 \cdot 10^{-4}$.

\begin{figure}[t]
\centering
\includegraphics[width=\linewidth]{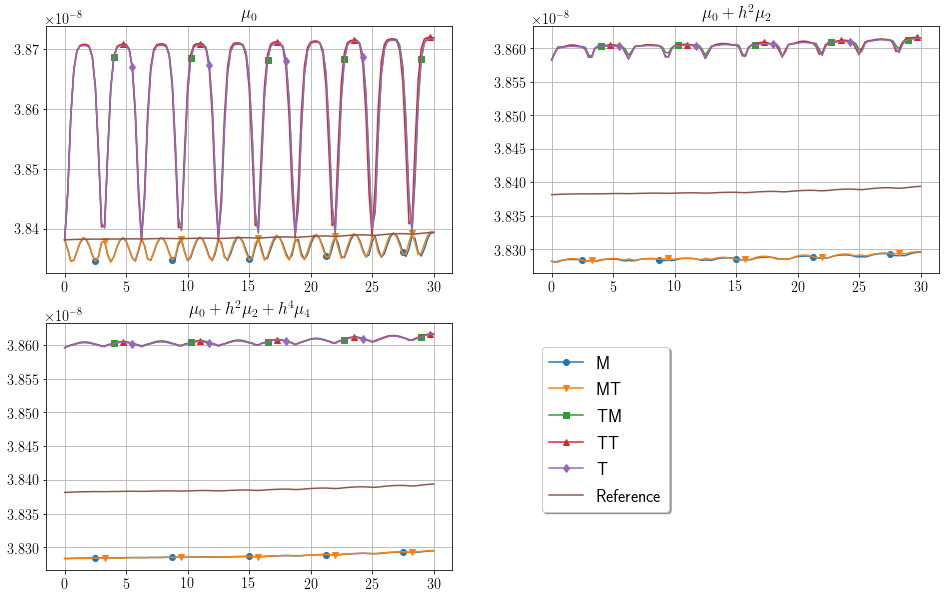}
\caption{
Evolution of the volume of a spherical cloud of $120$ points, arranged in a $600$-cell, with radius $0.01$ and centered around $(x,y,p_x,p_y) = (0,0,1,1)$, using our proposed  method $\Phi_h^{(4)}$ with time step  $h=0.25$ for the five variational integrators in \eqref{Ldisc} and a high-accuracy reference  solution. The volume is computed with respect to the measure indicated at the top of each graph.
}
\label{fig:measures}
\end{figure}

\begin{figure}[t]
\includegraphics[width=\linewidth]{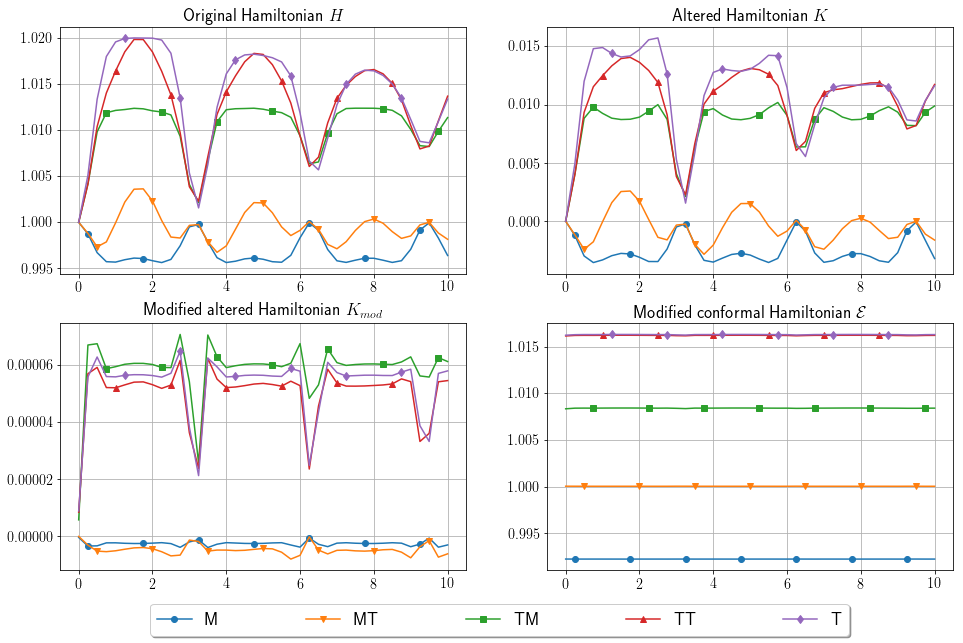}
\caption{
Application
of $\Phi_h^{(4)}$ to the nonholonomic particle in a harmonic potential
for each  discretization in~\eqref{Ldisc}: overview of the numerical values of the different energy functions involved in the algorithm. The initial condition is $(x,y,p_x),p_y) = (0,0,1,1)$ and the time step $h=0.25$.  }
\label{fig:hamiltonians}
\end{figure}

\begin{figure}[t]
\centering
\includegraphics[width=\linewidth]{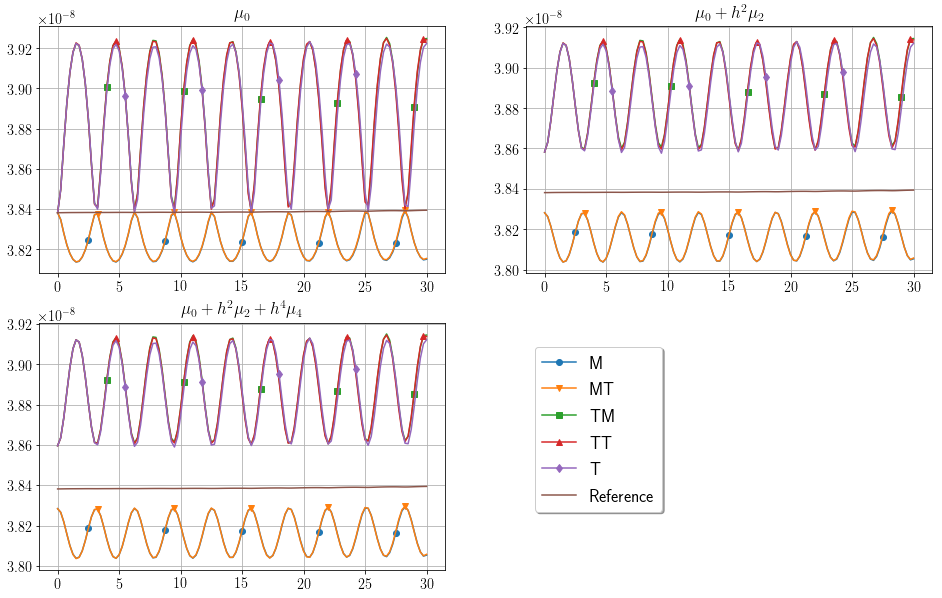}
\caption{
Evolution of the volume of a spherical cloud of $120$ points, arranged in a $600$-cell, with radius $0.01$ and centered around $(x,y,p_x,p_y) = (0,0,1,1)$, using the method $\Phi_h^{(0)}$ proposed in~\cite{hairer1997variable,reich1999backward,FernandezBlochOlver2012}  with time step $h=0.25$ for the five variational integrators in \eqref{Ldisc} and a high-accuracy reference  solution. The volume is computed with respect to the measure indicated at
the top of each graph.} 
\label{fig:measure-naive}
\end{figure}

\subsubsection{ Energy behavior}
\label{sss:energy}

In Figure~\ref{fig:hamiltonians} we graph the various  energy functions used in our algorithm and described in Table~\ref{tab:Hamiltonians}.  As before, we use the map $\Phi_h^{(4)}$ 
with the same time step $h=0.25$ and initial condition $(x,y,p_x,p_y) = (0,0,1,1)$. 
 The numerical values of the original Hamiltonian $H$ (top-left) oscillate close to its true value of 1. The altered Hamiltonian  $K_E$ with $E=1$, (top-right) shows a visually similar behavior, but the values are now close to 0, reflecting the fact that on the exact solution  $K_E$ is identically zero. 

The modified altered Hamiltonian $K_{mod}$ should be exactly preserved on the numerical solutions, up to a truncation error. The bottom-left panel of Figure \ref{fig:hamiltonians} shows the truncation  $K_{mod}^{(4)} = K_{mod}^{(5)}$, so we expect an error of order $h^6 \approx 2 \cdot 10^{-4}$. The graph indeed shows values of this order of magnitude 
 and smaller. In the bottom-right panel of Figure \ref{fig:hamiltonians} we plot the  truncated modified conformal Hamiltonian $\mathcal{E}^{(4)} = \mathcal{E}^{(5)}$. It exhibits oscillations of a similar size as  $K_{mod}^{(4)}$, which are invisible on the scale of this plot.

\subsubsection[Comparison of phi4 with phi0]{Comparison of our integrator $\Phi^{(4)}_h$
with the integrator  $\Phi^{(0)}_h$ used in~\cite{hairer1997variable,reich1999backward,FernandezBlochOlver2012} }
\label{ss:comparison}

When we apply the  measure preservation experiment described in Section~\ref{sss:measure-harmonic} to the integrator $\Phi^{(0)}_h$, we lose 
 the nice performance that was observed for $\Phi^{(4)}_h$. Indeed, 
in Figure~\ref{fig:measure-naive} we 
plot the evolution under $\Phi^{(0)}_h$, with $h=0.25$ as before, of the volume of the same test point cloud as before, with respect to the measures $\mu_0$, $\mu_{mod}^{(2)}$ and $\mu_{mod}^{(4)}$.
In contrast with Figure~\ref{fig:measures}, we observe large oscillations that do not seem to diminish with the order of approximation of the modified 
measure. 

On the other hand, experiments show little difference in 
the performance of $\Phi^{(0)}_h$ and $\Phi^{(4)}_h$ when
it comes to the error of approximation of the solutions and  energy behavior. We suspect that the reason
is that the solutions of the
system are bounded, due to the presence of the harmonic potential, and hence the conformal factor $\mathcal{N}$ is bounded away from zero along them. Therefore, since
the altered Hamiltonian $K_E = \mathcal{N}(H-E)$ is approximately conserved, we
expect that the difference $H-E$ remains small.  This good performance
properties of $\Phi^{(0)}_h$ do not hold for the free nonholonomic particle
treated below.

\subsection{Free nonholonomic particle}
\label{ss:numerics-free}

In order to further illustrate the benefits of our integrator $\Phi_h^{(\ell)}$ with respect to the discretization $\Phi_h^{(0)}$ used in~\cite{hairer1997variable,reich1999backward,FernandezBlochOlver2012}, 
 we treat  the nonholonomic particle in the
absence of potential energy. The main difference with respect to
the system with the harmonic potential treated in the previous section
is that the conformal
factor $\mathcal{N}$ approaches zero  as time grows along all solutions to the system for which
$p_y(0)\neq 0$. One may easily deduce this property from the equations
of motion. Because of this feature of the system, taking $\ell>0$ 
is not only relevant for measure preservation of  $\Phi_h^{(\ell)}$ but
it is also important 
in both  the energy behavior and  the overall numerical error, as we will see below.

In Figures \ref{fig:free2} and \ref{fig:free-naive2} we graph numerical values of the Hamiltonians and the Euclidean norm of the error in position and momentum. In Figure \ref{fig:free2} this is done with  $\Phi_h^{(4)}$,\ whereas in Figure \ref{fig:free-naive2} we use  $\Phi_h^{(0)}$.
In both cases we present the implementation of each of the five integrators in \eqref{Ldisc} with the same initial values and step size\footnote{The implementation of $\Phi_h^{(4)}$ requires the calculation of the fourth order truncation of $\mathcal{E}$. The second order terms 
$\mathcal{E}_2(q,p)$  for each  discretization in~\eqref{Ldisc} can be found in Appendix~\ref{sec:free-formulas} (where we also give 
the corresponding terms $K_2$ and  $\mathcal{N}_{2}$).
We do not present  the fourth order terms because of their complexity but they may be
found in our code \cite{vermeeren2020code}. }.
It is clear from the graphs that the fourth order 
method   $\Phi_h^{(4)}$ outperforms $\Phi_h^{(0)}$ 
both in the energy behavior and  numerical error.

\begin{figure}[p]
\includegraphics[width=\linewidth]{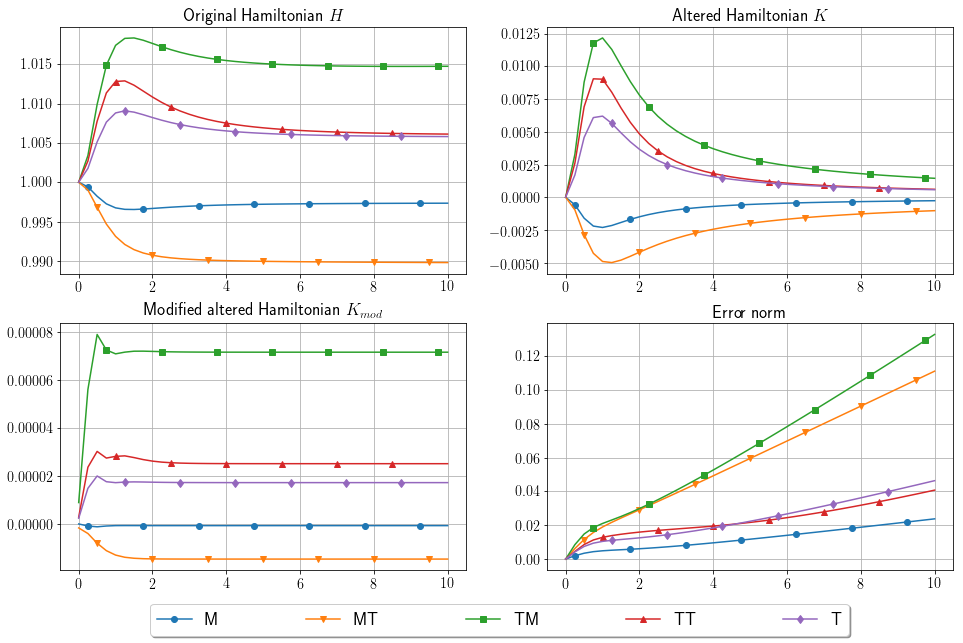}
\vspace{-6mm}\hbox{}
\caption{
Application
of $\Phi_h^{(4)}$ to the free nonholonomic particle
for each  discretization in~\eqref{Ldisc}. Overview of the numerical values of the different energy functions involved in the algorithm and error norm 
of the solutions. The initial condition is $(x,y,p_x,p_y) = (0,0,1,1)$ and the step size $h=0.25$. }
\label{fig:free2}
\end{figure}

\begin{figure}[p]
\centering
\includegraphics[width=\linewidth]{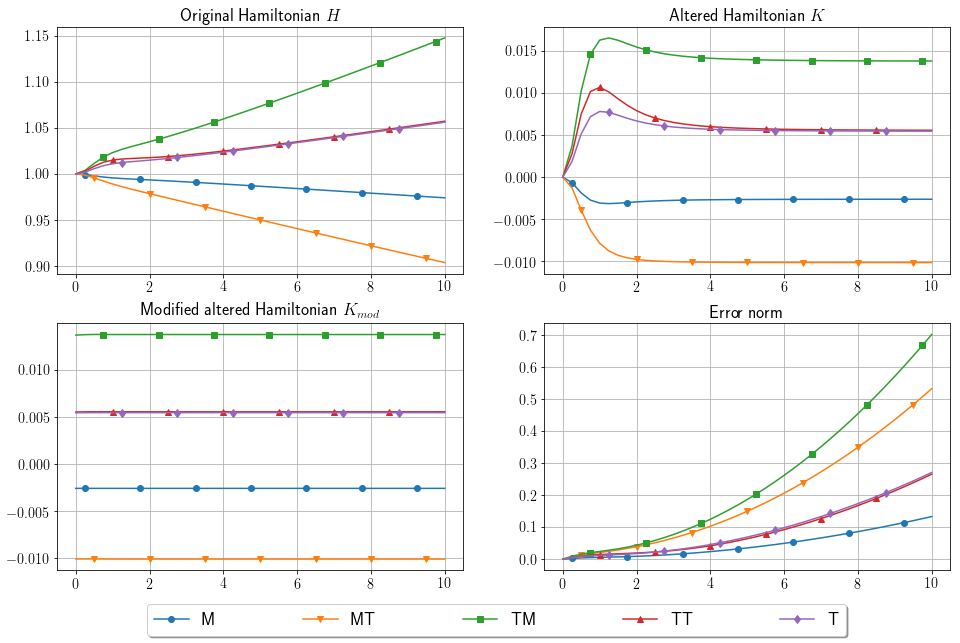}
\vspace{-6mm}\hbox{}
\caption{
Application
of $\Phi_h^{(0)}$ to the free nonholonomic particle
for each  discretization in~\eqref{Ldisc}: overview of the numerical values of the different energy functions involved in the algorithm and error norm 
of the solutions. The initial condition and the time step coincide with
those of Figure~\ref{fig:free2}. }
\label{fig:free-naive2}
\end{figure}

In Figure \ref{fig:free2} we see that for $\Phi_h^{(4)}$ the modified altered Hamiltonian is very close to zero and the altered Hamiltonian seems to converge to zero as well. As a consequence, we see very little drift in the Hamiltonian $H$. In contrast to the good energy behavior of $\Phi_h^{(4)}$,  
Figure~\ref{fig:free-naive2} shows that for $\Phi_h^{(0)}$ 
the value of the modified altered Hamiltonian is not so close to zero and the altered Hamiltonians $K_E=\mathcal{N}(H-E)$ approach the same nonzero values. Since $\mathcal{N}$ converges to zero, the difference 
$H-E$ grows leading to the observed 
drift in the Hamiltonian $H$. This poor behavior seems to carry over
to produce a faster growth in the error of the numerical approximation.
The poor energy behavior of $\Phi_h^{(0)}$ had  already been reported in 
\cite[Section 5.1]{FernandezBlochOlver2012}.

\section{Conclusions}
\label{s:conclusions}

We have introduced a discretization of the 
conformally Hamiltonian system \eqref{eq:itro-conformal} that is shown
to be formally interpolated by the flow of
a modified conformally Hamiltonian system, and in particular is measure preserving. Our 
discretization is implemented by 
applying a symplectic integrator to 
the altered system~\eqref{eq:itro-conformal-altered} where the parameter $E$ is taken as
the initial value of the modified conformal Hamiltonian $\mathcal{E}$ introduced in
Definition~\ref{def:mod-conf-Ham}. We have
conducted numerical experiments to compare our  approach  with the one followed by previous references~\cite{hairer1997variable,reich1999backward,FernandezBlochOlver2012}
where  the parameter $E$ in~\eqref{eq:itro-conformal-altered}
is instead taken as
the initial value of the  Hamiltonian $H$.
Our numerical results show that our method outperforms
the existing  one in measure preservation, energy behavior and overall numerical error. 

Moreover, to the best of our knowledge, the
application of our method to Hamiltonizable
Chaplygin systems provides the first example
of a measure preserving discretization of a 
measure preserving nonholonomic system.

\section*{Acknowledgements} 
The authors are grateful to Yuri Suris for inspiring discussions and insightful suggestions, which deeply influenced this paper.

LGN is thankful to the Alexander von Humboldt Foundation for a Georg Forster Advanced Research Fellowship that funded a research visit to TU Berlin where this project was
started. He also acknowledges support of the Program UNAM-DGAPA-PAPIIT IN115820 for his
research.

MV is funded currently by DFG Research Fellowship VE 1211/1-1 and at the time this work was started by the SFB Transregio 109 ``Discretization in Geometry and Dynamics''.

\appendix
\section{Second order terms used in the numerics}
\label{appendix}

The truncation after the second order term of the modified altered Hamiltonian for each discretization in \eqref{Ldisc} takes the form
\begin{equation*}
 K_{mod}^{(2)} (x,y, p_x, p_y;E)= \frac{1}{\sqrt{1+y^2}}(H(x,y,p_x,p_y)-E)+ h^2 K_2(x,y,p_x,p_y;E), 
\end{equation*}
where the function $K_2$ depends on the discretization and on the potential. Explicit expressions for $K_2$ for each discretization for the harmonic potential ($U(x,y)=\frac{1}{2}(x^2+y^2)$) and the free nonholonomic particle ($U(x,y)=0$) are listed below in \ref{sec:harmonic-formulas} and~\ref{sec:free-formulas}.
Using these expressions we can determine the truncated modified conformal Hamiltonian 
\[
\mathcal{E}^{(2)} (x,y, p_x,  p_y) = H(x,y,p_x,p_y)
+ h^2 \mathcal{E}_2 (x,y,p_x,p_y),
\]
where the functions $\mathcal{E}_2$, depending on the discretization, are also listed below. Finally, knowing $\mathcal{E}^{(2)}$ we can determine
\[\mathcal{N}_{mod}^{(2)} (x,y,p_x,p_y)=\frac{1}{\sqrt{1+y^2}} + h^2 \mathcal{N}_2(x,y,p_x,p_y) . \]
The specific form of $\mathcal{N}_2$ depending on the
discretization, can be found in \ref{sec:harmonic-formulas} and~\ref{sec:free-formulas} too. For both the harmonic  and
the free particle potentials we found that 
\begin{align*}
\mathcal{N}^{(2)}_{M} =
\mathcal{N}^{(2)}_{MT}, \qquad 
\mathcal{N}^{(2)}_{TM} = 
\mathcal{N}^{(2)}_{TT} = 
\mathcal{N}^{(2)}_{T}.
\end{align*}
In particular, we see that among these examples, $\mathcal{N}^{(2)}$ only depends on the discretization of the conformal factor $\mathcal{N}$. However, at higher orders also the discretization of the Lagrangian $L$ plays a role: each of the five discretizations \eqref{Ldisc} leads to a different formula for $\mathcal{N}_4$ and hence for $\mathcal{N}^{(\ell)}$ for all $\ell \geq 4$.

The explicit expressions below were obtained using the SageMath software, with code that is available at \cite{vermeeren2020code}. The same code can also be used to obtain higher-order terms.

\subsection{In a harmonic potential}
\label{sec:harmonic-formulas}

\begin{flalign*}
&(K_2)_M(x,y,p_x,p_y;E) &\\
&= \frac{1}{96 \left( {y}^{2}+1 \right) ^{5/2}} &\\
&\qquad \bigg({\left(3 p_y^{4} + 2 p_y^{2} - 1\right)} y^{6} + 2 {\left(2 p_y^{4} - {\left(3 {p_x}^{2} - 6 E - 4\right)} p_y^{2} - {\left(3 p_y^{2} - 1\right)} {x}^{2} + {p_x}^{2} - 2 E - 2\right)} y^{4} &\\
&\qquad - {\left({p_x}^{4} + p_y^{4} + {x}^{4} - 4 E {p_x}^{2} + 2 {\left(2 {p_x}^{2} - 4 E - 1\right)} p_y^{2} + 2 {\left({p_x}^{2} + 2 p_y^{2} - 2 E\right)} {x}^{2} + 4 E^{2} + 8 E + 4\right)} y^{2} &\\
&\qquad - 2 p_y^{4} + 2 {\left({p_x}^{2} - 2 E - 2\right)} p_y^{2} + 2 {\left(p_y^{2} - 2\right)} {x}^{2} - 4 {p_x}^{2}\bigg) &
\end{flalign*}
\begin{flalign*}
& (K_2)_{MT}(x,y,p_x,p_y;E) \\
& = \frac{1}{96 \left( {y}^{2}+1 \right) ^{5/2}} \\
&\qquad \bigg({\left(3 p_y^{4} + 14 p_y^{2} - 1\right)} y^{6} + 2 {\left(2 p_y^{4} - {\left(9 {p_x}^{2} - 6 E - 22\right)} p_y^{2} - {\left(3 p_y^{2} - 1\right)} {x}^{2} + {p_x}^{2} - 2 E - 2\right)} y^{4} \\
&\qquad - {\left({p_x}^{4} + p_y^{4} + {x}^{4} - 4 {\left(E + 3\right)} {p_x}^{2} + 2 {\left(14 {p_x}^{2} - 4 E - 19\right)} p_y^{2} + 2 {\left({p_x}^{2} + 2 p_y^{2} - 2 E\right)} {x}^{2} + 4 E^{2} + 8 E + 4\right)} y^{2} \\
&\qquad - 2 p_y^{4} - 2 {\left(5 {p_x}^{2} + 2 E - 4\right)} p_y^{2} + 2 {\left(p_y^{2} - 2\right)} {x}^{2}  + 8 {p_x}^{2}\bigg)
\end{flalign*}
\begin{flalign*}
&(K_2)_{TM}(x,y,p_x,p_y;E) &\\
& = \frac{-1}{96 \left( {y}^{2}+1 \right) ^{5/2}} &\\
&\qquad \bigg({\left(9 p_y^{4} - 14 p_y^{2} + 1\right)} y^{6} + 2 {\left(7 p_y^{4} + {\left(9 {p_x}^{2} + 6 E - 7\right)} p_y^{2} - {\left(3 p_y^{2} + 1\right)} {x}^{2} - {p_x}^{2} + 2 E + 2\right)} y^{4} &\\
&\qquad + {\left({p_x}^{4} + p_y^{4} + {x}^{4} - 4 E {p_x}^{2} + 2 {\left(5 {p_x}^{2} + 2 E + 2\right)} p_y^{2} + 2 {\left({p_x}^{2} - p_y^{2} - 2 E\right)} {x}^{2} + 4 E^{2} + 8 E + 4\right)} y^{2} &\\
&\qquad - 4 p_y^{4} - 4 {\left(2 {p_x}^{2} + 2 E - 1\right)} p_y^{2} + 4 {\left(p_y^{2} + 1\right)} {x}^{2}  + 4 {p_x}^{2}\bigg) &
\end{flalign*}
\begin{flalign*}
&(K_2)_{TT}(x,y,p_x,p_y;E) &\\
&= \frac{-1}{96 \left( {y}^{2}+1 \right) ^{5/2}} &\\
&\qquad \bigg({\left(9 p_y^{4} - 26 p_y^{2} + 1\right)} y^{6} + 2 {\left(7 p_y^{4} + {\left(15 {p_x}^{2} + 6 E - 25\right)} p_y^{2} - {\left(3 p_y^{2} + 1\right)} {x}^{2} - {p_x}^{2} + 2 E + 2\right)} y^{4} &\\
&\qquad +{\left({p_x}^{4} + p_y^{4} + {x}^{4} - 4 {\left(E + 3\right)} {p_x}^{2} + 2 {\left(17 {p_x}^{2} + 2 E - 16\right)} p_y^{2} + 2 {\left({p_x}^{2} - p_y^{2} - 2 E\right)} {x}^{2} + 4 E^{2} + 8 E + 4\right)} y^{2} &\\
&\qquad - 4 p_y^{4} + 4 {\left({p_x}^{2} - 2 E - 2\right)} p_y^{2} + 4 {\left(p_y^{2} + 1\right)} {x}^{2}  - 8 {p_x}^{2}\bigg) &
\end{flalign*}
\begin{flalign*}
& (K_2)_{T}(x,y,p_x,p_y;E) &\\
&= \frac{-1}{96 \left( {y}^{2}+1 \right) ^{5/2}} &\\
&\qquad \bigg({\left(9 p_y^{4} - 2 p_y^{2} + 1\right)} y^{6} + 24 {p_x} p_y {x} y^{3} + 2 {\left(7 p_y^{4} + {\left(3 {p_x}^{2} + 6 E - 1\right)} p_y^{2} - {\left(3 p_y^{2} + 1\right)} {x}^{2} - {p_x}^{2} + 2 E + 2\right)} y^{4} &\\
&\qquad + {\left({p_x}^{4} + p_y^{4} + {x}^{4} - 4 {\left(E + 3\right)} {p_x}^{2} + 2 {\left(5 {p_x}^{2} + 2 E - 4\right)} p_y^{2} + 2 {\left({p_x}^{2} - p_y^{2} - 2 E\right)} {x}^{2} + 4 E^{2} + 8 E + 4\right)} y^{2} &\\
&\qquad - 4 p_y^{4} + 24 {p_x} p_y {x} y + 4 {\left({p_x}^{2} - 2 E - 2\right)} p_y^{2} + 4 {\left(p_y^{2} + 1\right)} {x}^{2} - 8 {p_x}^{2}\bigg) &
\end{flalign*}

\begin{flalign*}
& (\mathcal{E}_2)_{M} (x,y,p_x,p_y) = \frac{2  p_y^4 {y}^{4} + p_y^4 {y}^{2} + p_y^2 {y}^{4} - p_y^4 - {y}^{4} - {p_x}^{2} - p_y^2 - {x}^{2} - {y}^{2}}{24  {\left({y}^{2} + 1\right)}}
&\\
&(\mathcal{E}_2)_{MT} (x,y,p_x,p_y) = \frac{2  p_y^4 {y}^{4} - 3  {p_x}^{2} p_y^2 {y}^{2} + p_y^4 {y}^{2} + 4  p_y^2 {y}^{4} - 3  {p_x}^{2} p_y^2 - p_y^4 + 6  p_y^2 {y}^{2} - {y}^{4} + 2  {p_x}^{2} + 2  p_y^2 - {x}^{2} - {y}^{2}}{24  {\left({y}^{2} + 1\right)}}
&\\
&(\mathcal{E}_2)_{TM} (x,y,p_x,p_y) = -\frac{4  p_y^4 {y}^{4} + 6  {p_x}^{2} p_y^2 {y}^{2} + 2  p_y^4 {y}^{2} - p_y^2 {y}^{4} - 3  {p_x}^{2} p_y^2 - 2  p_y^4 + {y}^{4} + {p_x}^{2} + p_y^2 + {x}^{2} + {y}^{2}}{24  {\left({y}^{2} + 1\right)}}
&\\
&(\mathcal{E}_2)_{TT} (x,y,p_x,p_y) = -\frac{4  p_y^4 {y}^{4} + 9  {p_x}^{2} p_y^2 {y}^{2} + 2  p_y^4 {y}^{2} - 4  p_y^2 {y}^{4} - 2  p_y^4 - 6  p_y^2 {y}^{2} + {y}^{4} - 2  {p_x}^{2} - 2  p_y^2 + {x}^{2} + {y}^{2}}{24  {\left({y}^{2} + 1\right)}}
&\\
&(\mathcal{E}_2)_{T} (x,y,p_x,p_y) = -\frac{4  p_y^4 {y}^{4} + 3  {p_x}^{2} p_y^2 {y}^{2} + 2  p_y^4 {y}^{2} + 2  p_y^2 {y}^{4} - 2  p_y^4 + 6  {p_x} {p_y} {x} {y} + {y}^{4} - 2  {p_x}^{2} - 2  p_y^2 + {x}^{2} + {y}^{2}}{24  {\left({y}^{2} + 1\right)}} &
\end{flalign*}

\begin{flalign*}
&(\mathcal{N}_2)_{M}(x,y,p_x,p_y) =
(\mathcal{N}_2)_{MT}(x,y,p_x,p_y) = -\frac{2  \left({p_y}^{2} - 1\right) {y}^{2} -  {p_y}^{2}}{24 \left({y}^{2} + 1\right)^\frac{3}{2}}
& \\
&(\mathcal{N}_2)_{TM}(x,y,p_x,p_y) = 
(\mathcal{N}_2)_{TT}(x,y,p_x,p_y) = 
(\mathcal{N}_2)_{T}(x,y,p_x,p_y) =  \frac{\left(2 p_y^{2} + 1\right) {y}^{2} - p_y^{2}}{12 \left({y}^{2} + 1\right)^\frac{3}{2}} &
\end{flalign*}

\subsection{For the free nonholonomic particle}
\label{sec:free-formulas}

\begin{flalign*}
(K_2)_M(x,y,p_x,p_y;E) 
= \frac{1}{96 \left( y^{2}+1 \right)^{5/2}} 
&\bigg( 3 p_y^{4} y^{6} + 2 \left(2 p_y^{4} - 3 \left(p_x^{2} - 2 E\right) p_y^{2}\right) y^{4} - 2 p_y^{4} + 2 \left(p_x^{2} - 2 E\right) p_y^{2} & \\
&\quad - \left(p_x^{4} + p_y^{4} - 4 E p_x^{2} + 4 \left(p_x^{2} - 2 E\right) p_y^{2} + 4 E^{2}\right) y^{2}\bigg) &
\end{flalign*}
\begin{flalign*}
(K_2)_{MT}(x,y,p_x,p_y;E)
= \frac{1}{96 \left( {y}^{2}+1 \right) ^{5/2}}
&\bigg(3 p_y^{4} y^{6} + 2 \left(2 p_y^{4} - 3 \left(3 p_x^{2} - 2 E\right) p_y^{2}\right) y^{4} - 2 p_y^{4} - 2 \left(5 p_x^{2} + 2 E\right) p_y^{2} & \\
&\quad - \left(p_x^{4} + p_y^{4} - 4 E p_x^{2} + 4 \left(7 p_x^{2} - 2 E\right) p_y^{2} + 4 E^{2}\right) y^{2}\bigg) &
\end{flalign*}
\begin{flalign*}
(K_2)_{TM}(x,y,p_x,p_y;E) 
 = \frac{-1}{96 \left( {y}^{2}+1 \right) ^{5/2}} 
&\bigg(9 p_y^{4} y^{6} + 2 \left(7 p_y^{4} + 3 \left(3 p_x^{2} + 2 E\right) p_y^{2}\right) y^{4} - 4 p_y^{4} - 8 \left(p_x^{2} + E\right) p_y^{2} & \\
&\quad + \left(p_x^{4} + p_y^{4} - 4 E p_x^{2} + 2 \left(5 p_x^{2} + 2 E\right) p_y^{2} + 4 E^{2}\right) y^{2}\bigg) &
\end{flalign*}
\begin{flalign*}
(K_2)_{TT}(x,y,p_x,p_y;E)
= \frac{-1}{96 \left( {y}^{2}+1 \right) ^{5/2}} 
& \bigg(9 p_y^{4} y^{6} + 2 \left(7 p_y^{4} + 3 \left(5 p_x^{2} + 2 E\right) p_y^{2}\right) y^{4} - 4 p_y^{4} + 4 \left(p_x^{2} - 2 E\right) p_y^{2} + & \\
&\quad \left(p_x^{4} + p_y^{4} - 4 E p_x^{2} + 2 \left(17 p_x^{2} + 2 E\right) p_y^{2} + 4 E^{2}\right) y^{2}\bigg) &
\end{flalign*}
\begin{flalign*}
(K_2)_{T}(x,y,p_x,p_y;E) 
= \frac{-1}{96 \left( {y}^{2}+1 \right) ^{5/2}} 
& \bigg(9 p_y^{4} y^{6} + 2 \left(7 p_y^{4} + 3 \left(p_x^{2} + 2 E\right) p_y^{2}\right) y^{4} - 4 p_y^{4} + 4 \left(p_x^{2} - 2 E\right) p_y^{2} + &\\
&\quad \left(p_x^{4} + p_y^{4} - 4 E p_x^{2} + 2 \left(5 p_x^{2} + 2 E\right) p_y^{2} + 4 E^{2}\right) y^{2}\bigg) &
\end{flalign*}

\begin{flalign*}
& (\mathcal{E}_2)_{M} (x,y,p_x,p_y) = \frac{1}{12} p_y^{4} y^{2} - \frac{1}{24} p_y^{4} &
\\
& (\mathcal{E}_2)_{MT} (x,y,p_x,p_y) = \frac{1}{12} p_y^{4} y^{2} - \frac{1}{8} p_x^{2} p_y^{2} - \frac{1}{24} p_y^{4} &
\\
& (\mathcal{E}_2)_{TM} (x,y,p_x,p_y) = -\frac{4 p_y^{4} y^{4} - 3 p_x^{2} p_y^{2} - 2 p_y^{4} + 2 \left(3 p_x^{2} p_y^{2} + p_y^{4}\right) y^{2}}{24 \left(y^{2} + 1\right)} &
\\
& (\mathcal{E}_2)_{TT} (x,y,p_x,p_y) = -\frac{4 p_y^{4} y^{4} - 2 p_y^{4} + \left(9 p_x^{2} p_y^{2} + 2 p_y^{4}\right) y^{2}}{24 \left(y^{2} + 1\right)} &
\\
&(\mathcal{E}_2)_{T} (x,y,p_x,p_y) = -\frac{4 p_y^{4} y^{4} - 2 p_y^{4} + \left(3 p_x^{2} p_y^{2} + 2 p_y^{4}\right) y^{2}}{24 \left(y^{2} + 1\right)} &
\end{flalign*}

\begin{flalign*}
& (\mathcal{N}_2)_{M}(x,y,p_x,p_y) =
(\mathcal{N}_2)_{MT}(x,y,p_x,p_y) = -\frac{2 \, p_y^{2} {y}^{2} - p_y^{2}}{24 \, {\left({y}^{2} + 1\right)}^{\frac{3}{2}}} &
\\
& (\mathcal{N}_2)_{TM}(x,y,p_x,p_y) = 
(\mathcal{N}_2)_{TT}(x,y,p_x,p_y) = 
(\mathcal{N}_2)_{T}(x,y,p_x,p_y) = \frac{2 p_y^{2} {y}^{2} - p_y^{2}}{12 \left({y}^{2} + 1\right)^\frac{3}{2}} &
\end{flalign*}

\bibliographystyle{abbrvnat_mv} % This is my personal style with hyperlinks.
%For standard options see https://www.overleaf.com/learn/latex/Natbib_bibliography_styles

\bibliography{refs}

\begin{thebibliography}{40}
\providecommand{\natexlab}[1]{#1}
\providecommand{\url}[1]{\texttt{#1}}
\expandafter\ifx\csname urlstyle\endcsname\relax
  \providecommand{\doi}[1]{doi: #1}\else
  \providecommand{\doi}{doi: \begingroup \urlstyle{rm}\Url}\fi

\bibitem[Bates and Śniatycki(1993)]{BS93}
Bates~L. \& Śniatycki~J.
\newblock \textit{Nonholonomic reduction}.
\newblock \href{http://dx.doi.org/10.1016/0034-4877(93)90073-N}{Reports on
  Mathematical Physics, 32\,:\,\mbox{99--115}}, 1993.

\bibitem[Benettin and Giorgilli(1994)]{benettin1994hamiltonian}
Benettin~G. \& Giorgilli~A.
\newblock \textit{On the {H}amiltonian interpolation of near-to-the identity
  symplectic mappings with application to symplectic integration algorithms}.
\newblock \href{http://dx.doi.org/10.1007/BF02188219}{Journal of Statistical
  Physics, 74\,:\,\mbox{1117--1143}}, 1994.

\bibitem[Bloch et~al.(2007)Bloch, Krishnaprasad, Marsden, and Murray]{BKMM}
Bloch~A.~M., Krishnaprasad~P., Marsden~J.~E. \& Murray~R.~M.
\newblock \textit{Nonholonomic mechanical systems with symmetry}.
\newblock \href{http://dx.doi.org/10.1007/s11202-007-0004-6}{Archive for
  Rational Mechanics and Analysis, 136\,:\,\mbox{21--99}}, 2007.

\bibitem[Borisov and Mamaev(2008{\natexlab{a}})]{Borisov2007}
Borisov~A.~V. \& Mamaev~I.~S.
\newblock \textit{Isomorphism and {H}amilton representation of some
  nonholonomic systems}.
\newblock \href{http://dx.doi.org/10.1134/S1560354708050079}{Siberian
  Mathematical Journal, 48\,:\,\mbox{26--36}}, 2008{\natexlab{a}}.

\bibitem[Borisov and Mamaev(2008{\natexlab{b}})]{borisov2008conservation}
Borisov~A.~V. \& Mamaev~I.~S.
\newblock \textit{Conservation laws, hierarchy of dynamics and explicit
  integration of nonholonomic systems}.
\newblock \href{http://dx.doi.org/10.1134/S1560354708050079}{Regular and
  Chaotic Dynamics, 13\,:\,\mbox{443--490}}, 2008{\natexlab{b}}.

\bibitem[Borisov et~al.(2013)Borisov, Mamaev, and
  Bizyaev]{borisov2013hierarchy}
Borisov~A.~V., Mamaev~I.~S. \& Bizyaev~I.~A.
\newblock \textit{The hierarchy of dynamics of a rigid body rolling without
  slipping and spinning on a plane and a sphere}.
\newblock \href{http://dx.doi.org/10.1134/S1560354713030064}{Regular and
  Chaotic Dynamics, 18\,:\,\mbox{277--328}}, 2013.

\bibitem[Calleja et~al.(2013)Calleja, Celletti, and de~la Llave]{CCLl2013}
Calleja~R.~C., Celletti~A. \& de~la Llave~R.
\newblock \textit{A {KAM} theory for conformally symplectic systems: efficient
  algorithms and their validation}.
\newblock \href{http://dx.doi.org/10.1016/j.jde.2013.05.001}{Journal of
  Differential Equations, 255\,:\,\mbox{978--1049}}, 2013.

\bibitem[Cantrijn et~al.(2002)Cantrijn, Cort{\'e}s, De~Le{\'o}n, and
  De~Diego]{CaCoLeMa}
Cantrijn~F., Cort{\'e}s~J., De~Le{\'o}n~M. \& De~Diego~D.~M.
\newblock \textit{On the geometry of generalized {C}haplygin systems}.
\newblock \href{http://dx.doi.org/10.1017/S0305004101005679}{Mathematical
  Proceedings of the Cambridge Philosophical Society, 132\,:\,\mbox{323--351}},
  2002.

\bibitem[Chaplygin(2008)]{Chaplygin}
Chaplygin~S.~A.
\newblock \textit{On the theory of motion of nonholonomic systems. {T}he
  reducing-multiplier theorem}.
\newblock \href{http://dx.doi.org/10.1134/S1560354708040102}{Regular and
  Chaotic Dynamics, 13\,:\,\mbox{369--376}}, 2008.

\bibitem[Cort{\'{e}}s and Mart{\'{\i}}nez(2001)]{Cortes2001}
Cort{\'{e}}s~J. \& Mart{\'{\i}}nez~S.
\newblock \textit{Non-holonomic integrators}.
\newblock \href{http://dx.doi.org/10.1088/0951-7715/14/5/322}{Nonlinearity,
  14\,:\,\mbox{1365--1392}}, 2001.

\bibitem[Ehlers et~al.(2005)Ehlers, Koiller, Montgomery, and Rios]{Ehlers2005}
Ehlers~K., Koiller~J., Montgomery~R. \& Rios~P.~M.
\newblock \textit{Nonholonomic systems via moving frames: {C}artan equivalence
  and {C}haplygin {H}amiltonization}.
\newblock In \href{http://dx.doi.org/10.1007/0-8176-4419-9_4}{Marsden~J.~E. \&
  Ratiu~T.~S., editors, \emph{The Breadth of Symplectic and Poisson Geometry:
  Festschrift in Honor of Alan Weinstein}, pages \mbox{75--120}}.
  Birkh{\"a}user, Boston, MA, 2005.

\bibitem[Fedorov and Jovanovic(2004)]{FedJov}
Fedorov~Y.~N. \& Jovanovic~B.
\newblock \textit{Nonholonomic {LR} systems as generalized {C}haplygin systems
  with an invariant measure and flows on homogeneous spaces}.
\newblock \href{http://dx.doi.org/10.1007/s00332-004-0603-3}{Journal of
  Nonlinear Science, 14\,:\,\mbox{341--381}}, 2004.

\bibitem[Fedorov and Jovanovi{\'c}(2009)]{FedJov2}
Fedorov~Y.~N. \& Jovanovi{\'c}~B.
\newblock \textit{Hamiltonization of the generalized {V}eselova {LR} system}.
\newblock \href{http://dx.doi.org/10.1134/S1560354709040066}{Regular and
  Chaotic Dynamics, 14\,:\,\mbox{495--505}}, 2009.

\bibitem[Fedorov et~al.(2015)Fedorov, Garc\'ia-Naranjo, and Marrero]{Fed2015}
Fedorov~Y.~N., Garc\'ia-Naranjo~L.~C. \& Marrero~J.~C.
\newblock \textit{Unimodularity and preservation of volumes in nonholonomic
  mechanics}.
\newblock \href{http://dx.doi.org/10.1007/s00332-014-9227-4}{Journal of
  Nonlinear Science, 25\,:\,\mbox{203--246}}, 2015.

\bibitem[Fernandez et~al.(2009)Fernandez, Mestdag, and Bloch]{Fernandez2009}
Fernandez~O.~E., Mestdag~T. \& Bloch~A.~M.
\newblock \textit{A generalization of {C}haplygin's {R}educibility {T}heorem}.
\newblock \href{http://dx.doi.org/10.1134/S1560354709060033}{Regular and
  Chaotic Dynamics, 14\,:\,\mbox{635--655}}, 2009.

\bibitem[Fernandez et~al.(2012)Fernandez, Bloch, and
  Olver]{FernandezBlochOlver2012}
Fernandez~O.~E., Bloch~A.~M. \& Olver~P.~J.
\newblock \textit{Variational integrators for {H}amiltonizable nonholonomic
  systems}.
\newblock \href{http://dx.doi.org/10.3934/jgm.2012.4.137}{Journal of Geometric
  Mechanics, 4\,:\,\mbox{137}}, 2012.

\bibitem[Ferraro et~al.(2008)Ferraro, Iglesias, and de~Diego]{Ferraro_2008}
Ferraro~S., Iglesias~D. \& de~Diego~D.~M.
\newblock \textit{Momentum and energy preserving integrators for nonholonomic
  dynamics}.
\newblock \href{http://dx.doi.org/10.1088/0951-7715/21/8/009}{Nonlinearity,
  21\,:\,\mbox{1911--1928}}, 2008.

\bibitem[Ferraro et~al.(2015)Ferraro, Jim{\'{e}}nez, and
  de~Diego]{Ferraro_2015}
Ferraro~S., Jim{\'{e}}nez~F. \& de~Diego~D.~M.
\newblock \textit{New developments on the geometric nonholonomic integrator}.
\newblock \href{http://dx.doi.org/10.1088/0951-7715/28/4/871}{Nonlinearity,
  28\,:\,\mbox{871--900}}, 2015.

\bibitem[Garc{\'\i}a-Naranjo(2019{\natexlab{a}})]{GNRubberRouth2019}
Garc{\'\i}a-Naranjo~L.~C.
\newblock \textit{Hamiltonisation, measure preservation and first integrals of
  the multi-dimensional rubber {R}outh sphere}.
\newblock \href{http://dx.doi.org/10.2298/TAM190130004G}{Theor. Appl. Mech.,
  46\,:\,\mbox{65--88}}, 2019{\natexlab{a}}.

\bibitem[Garc{\'\i}a-Naranjo(2019{\natexlab{b}})]{LGN18}
Garc{\'\i}a-Naranjo~L.~C.
\newblock \textit{Generalisation of {C}haplygin’s reducing multiplier theorem
  with an application to multi-dimensional nonholonomic dynamics}.
\newblock \href{http://dx.doi.org/10.1088/1751-8121/ab15f8}{Journal of Physics
  A: Mathematical and Theoretical, 52\,:\,\mbox{205203}}, 2019{\natexlab{b}}.

\bibitem[Garc{\'\i}a-Naranjo and Marrero(2020)]{GNMarr2018}
Garc{\'\i}a-Naranjo~L.~C. \& Marrero~J.~C.
\newblock \textit{The geometry of nonholonomic {C}haplygin systems revisited}.
\newblock \href{http://dx.doi.org/10.1088/1361-6544/ab5c0a}{Nonlinearity,
  33\,:\,\mbox{1297}}, 2020.

\bibitem[Hairer(1997)]{hairer1997variable}
Hairer~E.
\newblock \textit{Variable time step integration with symplectic methods}.
\newblock \href{http://dx.doi.org/10.1016/S0168-9274(97)00061-5}{Applied
  Numerical Mathematics, 25\,:\,\mbox{219--227}}, 1997.

\bibitem[Hairer and Lubich(1997)]{hairer1997lifespan}
Hairer~E. \& Lubich~C.
\newblock \textit{The life-span of backward error analysis for numerical
  integrators}.
\newblock \href{http://dx.doi.org/10.1007/s002110050271}{Numerische Mathematik,
  76\,:\,\mbox{441--462}}, 1997.

\bibitem[Hairer et~al.(2006)Hairer, Lubich, and Wanner]{hairer2006geometric}
Hairer~E., Lubich~C. \& Wanner~G.
\newblock \textit{Geometric Numerical Integration: Structure-Preserving
  Algorithms for Ordinary Differential Equations}.
\newblock Springer, Berlin, 2nd edition, 2006.

\bibitem[Jovanovi{\'c}(2019)]{jovanovic2019}
Jovanovi{\'c}~B.
\newblock \textit{Note on a ball rolling over a sphere: integrable {C}haplygin
  system with an invariant measure without {C}haplygin {H}amiltonization}.
\newblock \href{http://dx.doi.org/10.2298/TAM190322003J}{Theoretical and
  Applied Mechanics, 46\,:\,\mbox{97--108}}, 2019.

\bibitem[Kobilarov et~al.(2010)Kobilarov, Marsden, and Sukhatme]{Kobilarov2010}
Kobilarov~M., Marsden~J.~E. \& Sukhatme~G.~S.
\newblock \textit{Geometric discretization of nonholonomic systems with
  symmetries}.
\newblock \href{http://dx.doi.org/10.3934/dcdss.2010.3.61}{Discrete \&
  Continuous Dynamical Systems - S, 3\,:\,\mbox{61}}, 2010.

\bibitem[Koiller(1992)]{Koi}
Koiller~J.
\newblock \textit{Reduction of some classical non-holonomic systems with
  symmetry}.
\newblock \href{http://dx.doi.org/10.1007/BF00375092}{Archive for Rational
  Mechanics and Analysis, 118\,:\,\mbox{113--148}}, 1992.

\bibitem[Leimkuhler and Reich(2004)]{leimkuhler2004simulating}
Leimkuhler~B. \& Reich~S.
\newblock \textit{Simulating Hamiltonian Dynamics}.
\newblock Volume~14 of \emph{Cambridge Monographs on Applied and Computational
  Mathematics}. Cambridge University Press, Cambridge, 2004.

\bibitem[Levi-Civita et~al.(1906)]{levi1906resolution}
Levi-Civita~T. et~al.
\newblock \textit{Sur la r{\'e}solution qualitative du probleme restreint des
  trois corps}.
\newblock Acta Mathematica, 30\,:\,\mbox{305--327}, 1906.

\bibitem[Marle(2012)]{Marle2012}
Marle~C.-M.
\newblock \textit{A property of conformally {H}amiltonian vector fields;
  {A}pplication to the {K}epler problem}.
\newblock \href{http://dx.doi.org/10.3934/jgm.2012.4.181}{Journal of Geometric
  Mechanics, 4\,:\,\mbox{181}}, 2012.

\bibitem[Marsden and West(2001)]{MarsdenWest}
Marsden~J.~E. \& West~M.
\newblock \textit{Discrete mechanics and variational integrators}.
\newblock \href{http://dx.doi.org/10.1017/S096249290100006X}{Acta Numerica,
  10\,:\,\mbox{357--514}}, 2001.

\bibitem[McLachlan and Perlmutter(2001)]{mclachlan2001conformal}
McLachlan~R. \& Perlmutter~M.
\newblock \textit{Conformal {H}amiltonian systems}.
\newblock \href{http://dx.doi.org/10.1016/S0393-0440(01)00020-1}{Journal of
  Geometry and Physics, 39\,:\,\mbox{276--300}}, 2001.

\bibitem[McLachlan and Perlmutter(2006)]{McPerlm2006}
McLachlan~R. \& Perlmutter~M.
\newblock \textit{Integrators for nonholonomic mechanical systems}.
\newblock \href{http://dx.doi.org/10.1007/s00332-005-0698-1}{J. Nonlinear Sci.,
  16\,:\,\mbox{283--328}}, 2006.

\bibitem[Modin and Verdier(2020)]{modin2017}
Modin~K. \& Verdier~O.
\newblock \textit{What makes nonholonomic integrators work?}
\newblock \href{http://dx.doi.org/10.1007/s00211-020-01126-y}{Numerische
  Mathematik, 145\,:\,\mbox{405–435}}, 2020.

\bibitem[Reich(1999)]{reich1999backward}
Reich~S.
\newblock \textit{Backward error analysis for numerical integrators}.
\newblock \href{http://dx.doi.org/10.1137/S0036142997329797}{SIAM Journal on
  Numerical Analysis, 36\,:\,\mbox{1549--1570}}, 1999.

\bibitem[Simoes et~al.(2020)Simoes, Marrero, and de~Diego]{alex2020}
Simoes~A.~A., Marrero~J.~C. \& de~Diego~D.~M.
\newblock \textit{Exact discrete {L}agrangian mechanics for nonholonomic
  mechanics}.
\newblock \href{https://arxiv.org/abs/2003.11362}{arXiv:2003.11362}, 2020.

\bibitem[Stanchenko(1989)]{Stanchenko}
Stanchenko~S.
\newblock \textit{Non-holonomic {C}haplygin systems}.
\newblock \href{http://dx.doi.org/10.1016/0021-8928(89)90126-3}{Journal of
  Applied Mathematics and Mechanics, 53\,:\,\mbox{11--17}}, 1989.

\bibitem[Vermeeren(2017)]{vermeeren2017modified}
Vermeeren~M.
\newblock \textit{Modified equations for variational integrators}.
\newblock \href{http://dx.doi.org/10.1007/s00211-017-0896-4}{Numerische
  Mathematik, 137\,:\,\mbox{1001--1037}}, 2017.

\bibitem[Vermeeren(2020)]{vermeeren2020code}
Vermeeren~M.
\newblock \textit{Support code for ``{S}tructure preserving discretization of
  time-reparametrized {H}amiltonian systems with application to nonholonomic
  mechanics''}.
\newblock
  \href{http://dx.doi.org/10.5281/zenodo.3988087}{DOI\;10.5281/zenodo.3988087}.
\newblock \url{https://github.com/mvermeeren/conf-ham-sys-2020}, 2020.

\bibitem[Veselov and Veselova(1988)]{veselov1988integrable}
Veselov~A.~P. \& Veselova~L.
\newblock \textit{Integrable nonholonomic systems on {L}ie groups}.
\newblock \href{http://dx.doi.org/10.1007/BF01158420}{Mathematical notes of the
  Academy of Sciences of the USSR, 44\,:\,\mbox{810--819}}, 1988.

\end{thebibliography}

\end{document}